\documentclass{amsart}

\usepackage{xypic}
\xyoption{curve}      %for curved arrows
\xyoption{cmtip}      %for cm arrowheads
\usepackage{amsmath, amssymb, amsthm}
\usepackage{fullpage, epsfig}

\input{NilesDefs.sty}
\input{NilesEnvs.sty}

\let\oldmarginpar\marginpar
\renewcommand\marginpar[1]{\-\oldmarginpar%[\raggedleft\footnotesize #1]%
{\raggedright\footnotesize #1}}

\usepackage{slashed}

\newcommand{\lng}{\left\langle}
\newcommand{\rng}{\right\rangle}
\DeclareMathOperator*{\colim}{colim}

\newcommand{\zcb}{\text{\fontfamily{pzc}\selectfont
B}_A
}

\newcommand{\wt}{\widetilde}

\title{Morita Theory For Derived Categories: A Bicategorical Perspective}
\author{Niles Johnson}
\address{University of Chicago\\
Department of Mathematics\\
5734 South University Ave.\\
Chicago, IL 60637}
\urladdr{http://www.math.uchicago.edu/~niles}
\email{niles@math.uchicago.edu}

\begin{document}

\begin{abstract}
  We present a bicategorical perspective on derived Morita theory for
  rings, DG algebras, and spectra.  This perspective draws a
  connection between Morita theory and the bicategorical Yoneda Lemma,
  yielding a conceptual unification of Morita theory in derived and
  bicategorical contexts.  This is motivated by study of Rickard's
  theorem for derived equivalences of rings and of Morita theory for
  ring spectra, which we present in Sections \ref{outline} and
  \ref{proof}.  Along the way, we gain an understanding of the
  barriers to Morita theory for DG algebras and give a conceptual
  explanation for the counterexample of Dugger and Shipley.
\end{abstract}
\maketitle

\section{Introduction}\label{intro}
A bicategorical perspective on Morita theory is rooted in the
observation that Morita theory is a theory of \emph{bi}modules, not
simply left modules or right modules.  To give an incomplete survey,
this perspective has yielded extensions to the theory of distributors
over enriched categories by Fisher-Palmquist and Palmquist
\cite{fisherpalmquist1975mce}, to subfactor theory by M{\"u}ger
\cite{muger2003sca}, to bialgebroids by Szlach{\'a}nyi
\cite{szlachanyi2004mme}, and to von Neumann algebras by Brouwer
\cite{brouwer2003bam}.

A largely disjoint body of work has studied Morita theory in derived
contexts.  This began with the work of Rickard, studying derived
categories of rings in \cite{rickard1989mtd} and
\cite{rickard1991ded}.  Rickard's results were re-treated by Schwede
in \cite{schwede2004mta} following work of Keller in
\cite{keller1994ddc}; cf. \cite[Ch. 8]{konig1998deg} for a very
readable overview.  Dugger, Schwede and Shipley give partial
extensions to ring spectra and differential graded algebras in
\cite{schwede2003smc}, \cite{dugger2007ted}, and related work.
Derived Morita theory for differential graded categories has been
studied by Keller in \cite{keller1994ddc} and To{\"e}n in
\cite{toen2007htd}.  The development of derived Morita theory has
required more delicacy than its bicategorical counterpart, with the
counterexample of \cite{dugger2007ted} being a barrier to expected
generalizations of Rickard's theorem.  The work of Dugger, Shipley,
and To{\"e}n is motivated in part by this unsatisfactory situation.

We present a conceptual unification of bicategorical Morita theory
with Morita theory for derived categories by developing Morita theory
in triangulated bicategories.  In \autoref{context}, we introduce
bicategorical language for those to whom it is unfamiliar, and in
\autoref{trbicat} we describe what is meant by a triangulated
bicategory.  From this vantage, we are able to give a conceptual
explanation of derived Morita theory for which the results (and
counterexample) of \cite{dugger2007ted} and \cite{toen2007htd} become
verifying examples.

This is achieved in three stages.  First, in \autoref{yoneda} we
remind the reader of the bicategorical Yoneda Lemma (\ref{YL}) and
explain that what is often called Morita theory is a corollary
(\ref{MoII}) of this Yoneda Lemma.  We encourage the intuition that
bicategorical Morita theory is as elementary as the bicategorical
Yoneda Lemma; this gives one possible reason that bicategorical
perspectives have yielded such an abundance of generalizations for
classical Morita theory.  Second, we modify a standard observation
from the context of enriched category theory to explain that, in
bicategories with left and right internal homs, Morita theory must
necessarily focus on equivalences which are enriched over the internal
homs (\ref{rtF}).  This gives a reason for the results and examples
mentioned above.  Finally, in \autoref{practical} we apply our
understanding of the Yoneda Lemma.  Our perspective allows us to
re-frame the issue of derived Morita theory and sheds some new light
on the subtleties there.  In \autoref{enrichment} we discuss the
relationship of Morita theory to ambient enrichments.  In classical
Morita theory the ambient abelian enrichment is automatically
preserved (\autoref{enrichmentRmk}), but this is not the case in all
other contexts.  This provides, for example, a reason why the
development of Morita theory has met unexpected barriers in the DG
case.

We foreshadow the bicategorical perspective by outlining a proof of
Rickard's theorem in \autoref{outline}.  After establishing
bicategorical terminology in \autoref{context}, \autoref{proof} gives
the details of this proof.  In \autoref{RiSp} we give a generalization
of this theorem to ring spectra.  The last two sections cover some
basic model-theoretic results for our bicategory of DG-algebras and
their bimodules; they are verifications that expected results from the
theory of monoidal stable model categories generalize to our context
in straightforward ways.  \autoref{Model} gives a bicategorical
development of the standard model structure for DG algebras, and
\autoref{base-change} describes the formal structure arising from
change of base algebra.

\section{Outline}\label{outline}
In this section, we demonstrate our perspective by giving Rickard's
theorem for derived Morita theory of rings, together with an outline
of its proof.  For reference, we give a statement of the classical
Morita theorem together with its proof, and we follow these with some
remarks about derived Morita theory for DGAs.  The counterexample of
\cite{dugger2007ted} is given as \autoref{DSeg}, and shows that
Rickard's theorem does not generalize to DGAs as stated.  We give
some hints about where this breakdown occurs, to be described more
fully after we have developed the appropriate bicategorical
perspective.

\begin{note}
In the following, we implicitly understand ``module'' to mean
``right-module'', unless it is otherwise qualified.  The most frequent
instance of this opposite qualification will be that the endomorphism
ring of a right-module acts on its \emph{left}, and vice-versa.
\end{note}

\begin{thm}[Rickard]\label{Ri}
  Let $k$ be a commutative ring, and let $R$ and $S$ be $k$-algebras.
  The derived categories $\sD_k(R)$ and $\sD_k(S)$ are equivalent as
  triangulated categories if and only if there is an object $T$ of
  $\sD_k(S)$ with the following three properties.
  \begin{enumerate}
    \renewcommand{\theenumi}{\emph{\roman{enumi}}}
    \item $T$ is (quasi-isomorphic to) a bounded complex of
      finitely-generated projective $S$-modules.
    \item $T$ generates the triangulated category $\sD_k(S)$.
    \item The graded endomorphism algebra $\sD_k(S)[T,T]_*$ is
    concentrated in degree zero and isomorphic to $R$ as a
    $k$-algebra.
  \end{enumerate}
\end{thm}

\begin{thm}[Morita]
  Let $R$ and $S$ be rings.  The categories $Mod_R$ and $Mod_S$ are
  equivalent if and only if there is an object $P$ of $Mod_S$ with the
  following three properties.
  \begin{enumerate}
    \renewcommand{\theenumi}{\emph{\roman{enumi}}}
    \item \label{cmi} $P$ is a finitely-generated projective $S$-module.
    \item \label{cmii} $P$ generates the abelian category $Mod_S$.
    \item \label{cmiii} The endomorphism ring $Hom_S(P,P)$ is
      isomorphic to $R$.
  \end{enumerate}
\end{thm}

\begin{proof}
For the classical theorem of Morita, we make use of the bicategory
$\sM$ of rings and their bimodules.  The dual basis lemma gives that
condition (\ref{cmi}) is equivalent to the canonical coevaluation map
\[
\nu: P \otimes_S \Hom_S(P,S) \to \Hom_S(P,P)
\] 
being an isomorphism.  With condition (\ref{cmiii}), this can be
phrased in the bicategorical context by saying that $(P, \Hom_S(P,S))$
form a dual pair over $S$ and $R$, which means that the functors $-
\otimes_R P$ and $- \otimes_S \Hom(P,S)$ are an adjoint pair.  The
generating condition, (\ref{cmii}), is equivalent to the canonical
evaluation map $\epz: \Hom_S(P,S) \otimes_R P \to S$ being an
isomorphism, and hence this dual pair is an invertible pair, giving an
adjoint equivalence of categories.  The converse is also easy to see
classically, since an equivalence of categories $F : Mod_R \to Mod_S$
induces an isomorphism on the morphisms between modules, and therefore
a ring isomorphism
\[ 
R \iso \Hom_R(R,R) \fto{F} \Hom_S(FR,FR).
\] 
Moreover, the other two properties are enjoyed by $R$ and preserved by
equivalences, so taking $P = FR$ gives the converse.
\end{proof}

\begin{rmk}\label{enrichmentRmk}
  For our future discussion, it is worth noting that this argument
  takes advantage of the elementary fact that the abelian group
  structure on $\Hom_R(R,R)$ is necessarily preserved by $F$.  In fact
  any left adjoint functor between abelian categories is automatically
  enriched over abelian groups.  This is neither expected nor true of
  more general enrichments.
\end{rmk}

This point of view on the classical theorem is readily generalized to
the proof of Rickard's theorem.  In order to clarify the proof, we
separate Rickard's theorem into a well-known lemma and two
propositions.

\begin{lem}\label{RiL}
  Let $E$ be a DG $k$-algebra whose homology is concentrated in degree
  zero.  Then $E$ is quasi-isomorphic to its homology, and hence there
  is a triangulated equivalence $\sD_k(E) \hty \sD_k(H_*E)$.
\end{lem}
\begin{proof}
  Let
  \[
  (E_+)_n =
  \begin{cases}
    E_n ,\, n > 0\\
    Z_0(E) = ker(d_0) ,\, n = 0\\
    0 ,\, n < 0
  \end{cases}
  \]
  Then the projection and inclusion define a zig-zag of
  quasi-isomorphisms
  \[\text{H}_0(E) \fot{\hty} E_+ \fto{\hty} E,\]
  and base change along these maps gives equivalences of derived
  categories.
\end{proof}

\begin{defn}[formality]
  DG algebras which are quasi-isomorphic to their homology are called
  formal.
\end{defn}

\begin{prop}\label{RiP2}
  Let $R$ and $S$ be $k$-algebras.  If $F: \sD_k(R) \hty \sD_k(S)$ is an
  equivalence of triangulated categories, then there is an object $T
  \elt \sD_k(S)$ with the following two properties:
  \begin{enumerate}
    \renewcommand{\theenumi}{\emph{\roman{enumi}}}
  \item $T$ is (quasi-isomorphic to) a bounded complex of
    finitely-generated projective $S$-modules.
  \item $T$ generates the triangulated category $\sD_k(S)$.
  \end{enumerate}
  Moreover, the DG endomorphism algebra $End_S(T)$ is
  quasi-isomorphic to $R$ as a DG $k$-algebra.
\end{prop}
\begin{proof}
  A common proof of this proposition (see \cite{schwede2004mta}, for
  example) is to remark that the two conditions are preserved by exact
  equivalences and are enjoyed by $R$ regarded as a module over
  itself, hence also $T = FR$ has these properties.

  Equivalences induce isomorphisms on homology of endomorphism DG
  $k$-algebras, and so $End_S(T)$ has homology which is concentrated
  in degree $0$ and isomorphic to the homology of $R$ (that is, $R$
  itself).  \autoref{RiL} shows that $End_S(T)$ is therefore formal,
  and hence $End_S(T)$ and $R$ are quasi-isomorphic DG $k$-algebras.
\end{proof}

Since $H_*End_S(T) = \sD_k (S)[T,T]_*$, this proves one implication in
Rickard's theorem.  The other implication is proved by again applying
\autoref{RiL} in the case $E = End_s(T)$.  If $R$ is isomorphic to
$H_*E$, then formality ensures that $R$ and $E$ are quasi-isomorphic
and hence $\sD_k(E) \hty \sD_k(R)$.  The following proposition then
proves this direction of Rickard's theorem, by specializing to the
case that $S$ is a DG $k$-algebra concentrated in degree $0$.

\begin{note}
Dualizable modules over a DG $k$-algebra are defined
in \autoref{dltybicat}, but for the current argument it is enough to
observe that when $S$ is a DG $k$-algebra concentrated in degree 0,
then a right-dualizable $S$-module is simply a bounded complex of
finitely generated and projective $S$-modules.
\end{note}

\begin{prop}\label{RiP1}
  Let $S$ be a DG $k$-algebra, and let $T$ be a DG $S$-module.  If $T$ has
  the following two properties, then $\sD_k(S)$ and
  $\sD_k(\text{End}_S(T))$ are equivalent as triangulated categories.
  \begin{enumerate}
    \renewcommand{\theenumi}{\emph{\roman{enumi}}}
  \item $T$ is a right-dualizable $S$-module.
  \item $T$ generates the triangulated category $\sD_k(S)$.
  \end{enumerate}
\end{prop}
The proof is given in \autoref{proof} below.  It shows that $T$ has a
dual and the desired equivalence is given by the derived tensor
product with $T$; its inverse is derived tensor with the dual of $T$.
Such equivalences are called \emph{standard derived equivalences}.

Since DG $k$-algebras are, in general, not formal, we do not expect
Rickard's theorem to generalize to DG $k$-algebras as stated.
However, if the third condition for $T$ is strengthened to a
requirement that $End_S(T)$ be quasi-isomorphic to $R$, then
\autoref{RiP1} is a proof for one direction.  \autoref{RiP2} is not
generally true when $R$ and $S$ are taken to be DG $k$-algebras, and
this is the main barrier to generalizing Rickard's theorem.  The
difficulty is that one does not have formality for DG $k$-algebras in
general.  More precisely, formality is used in the proof of
\autoref{RiP2} to show that an equivalence of derived categories (of
rings) is sufficient to guarantee a quasi-isomorphism of DG
$k$-algebras between a ring and the endomorphism DG $k$-algebra of its
image under the equivalence.  Such a quasi-isomorphism is neither
expected nor present in greater generality.  We investigate this in
\autoref{yoneda}, but for now we give an example to illustrate how
\autoref{RiP2} can fail in the DG situation.

\begin{eg}\label{DSeg}
  In \cite{dugger2007ted}, an example of two DG rings is given: $C =
  \bZ[e]/(e^4)$ with $|e| = 1$ and $d(e) = 2$, and $A = H_*C$.  The
  model categories of $C$-modules and $A$-modules are Quillen
  equivalent, but there is no possible bimodule with the properties
  listed in Rickard's theorem.  That there can be no such bimodule is
  proven by noting that $A$ is a DGA over $\bZ/2$, but $C$ is not
  quasi-isomorphic to any DGA over $\bZ/2$.  Since the endomorphism
  algebra of any $A$-module would also be a DGA over $\bZ/2$, there
  cannot be a $C$-$A$-bimodule with endomorphism DGA quasi-isomorphic
  to $C$.  The argument that these DGAs do have Quillen equivalent
  categories of modules involves a THH (topological Hochschild
  homology) calculation which produces an equivalence of
  $\bS$-algebras between their Eilenberg-Mac Lane spectra.  More
  details can be found in \cite{shipley2006mts}.
\end{eg}

To understand the force of this example better, we note that the
equivalences arising in \autoref{RiL} and \autoref{RiP1} are
\emph{standard derived equivalences}; they are given by derived tensor
with a DG-bimodule.  These are manifestly induced by Quillen
equivalences of model categories, namely the underived tensor on the
categories of DG-modules.  The example above shows, however, that the
property of being induced by a Quillen equivalence is not sufficient
to characterize the standard derived equivalences.  To reiterate, the
DGAs in the example do have Quillen equivalent module categories, but
the induced equivalence of derived categories \emph{cannot} be a
standard derived equivalence.

A key to fully characterizing standard derived equivalences is an
observation about the organized way in which standard derived
equivalences preserve \emph{bi}module structures.  If $M$ is an
$R$-$S$ DG-bimodule, then $- \otimes_R M$ preserves \emph{left}-module
structure for all right $R$-modules.  By neglect, this can be regarded
as a functor from right DG $R$-modules to right DG $S$-modules, but to
do so forgets too much.  In the example above, the Quillen equivalence
of right-module categories does not preserve \emph{left}-module
structure, so it cannot induce a standard derived equivalence.

An alternative perspective might point out that the standard derived
equivalences also preserve categorical enrichment.  That is, with $M$
as above, $- \otimes_R M$ induces morphisms of hom \emph{objects}, and
is compatible with the enriched composition in the expected way.  The
Quillen pair of functors produced in the example of
\cite{dugger2007ted} is not a pair of DG-enriched functors.

The point of \autoref{rtF} is that these two perspectives are in fact
equivalent.  Moreover, \autoref{MoII} interprets the Yoneda Lemma
(\ref{YL}) as a statement that these (equivalent) properties do
characterize the standard derived equivalences.  These observations
are unlikely to be surprising to an enriched category theorist, as
they are the apparent generalizations (or specializations) of standard
results to our context, but they have been included for the algebraist
or topologist who may be unfamiliar with this perspective.

In \cite[Ch. 8]{konig1998deg}, Keller remarks that there are no known
examples of \textit{non-standard} derived equivalences for rings.  Our
characterization of Morita theory via the Yoneda Lemma yields the
following proposition.  The notation $\sD_k(X,Y)$ denotes derived
categories of bimodules, described in more detail below.  Also note
that $\End_k(A)$ is taken to mean the derived endomorphism ring.
\begin{prop*}[See \ref{StdDerEquivs}]
  Let $k$ be a commutative ring, let $A$ be a DG $k$-algebra and let
  $f : \sD_k(A) \to \sD_k(\End_k(A))$ be an equivalence of triangulated
  categories.  Then $f$ is a standard derived equivalence if and only
  if the following conditions hold.
  \begin{enumerate}
    \renewcommand{\theenumi}{\textit{\roman{enumi}}}
  \item The equivalence given by $f$ is an enriched equivalence.
  \item There is an enriched equivalence $f' : \sD_k(A,A) \to \sD_k(\End_k(A), A)$.
  \item The two equivalences, $f$ and $f'$ are compatible in the
    following sense: If $T', U' \elt \sD_k(A,A)$ and $T,U \elt
    \sD_k(A) = \sD_k(A,k)$, then there are natural maps 
    \begin{center}
    \begin{tabular}{ll}
    $\Ext_A(T,U') \to \Ext_{\End_k(A)}(fT,f'U')$ & in $\sD_k(k,A)$\\
    $\Ext_A(T',U) \to \Ext_{\End_k(A)}(f'T',fU)$ & in $\sD_k(A,k)$
    \end{tabular}
    \end{center}
    which commute with the pairing induced by composition.  (That is,
    the squares in \autoref{Ermk} commute.)
  \end{enumerate}
\end{prop*}

Specializing to the case that $A$ and $\Ext_k(A,A)$ are concentrated
in degree $0$, this proposition implies that a derived equivalence of
rings is standard if and only if it preserves bimodule structure as
described above; see \autoref{rtF}.

To make the characterization of standard derived equivalences clear,
we cannot avoid introducing bicategorical language.  In particular, we
must describe the notion of \emph{pseudofunctor}, especially
\emph{represented pseudofunctor}, and \emph{strong transformation} of
(represented) pseudofunctor.  This language is relevant because a
\emph{component} of a transformation between pseudofunctors is a
functor between certain categories, and the question of whether a
given derived equivalence is a standard derived equivalence is
precisely the same as whether the given functor is a component of a
strong transformation between two specific pseudofunctors.  We address
this fully in \autoref{yoneda}, but we begin in \autoref{context} by
introducing our bicategorical context.  In \autoref{proof} we
illustrate the bicategorical language with a proof of \autoref{RiP1},
and in \autoref{Model} and \autoref{base-change} we give a further
development of the structure present in our bicategorical framework.
This is the foundation for our applications of the Yoneda Lemma in
\autoref{yoneda}.

\section{Bicategorical Context}\label{context}
We make use of a bicategorical context to organize and clarify our
understanding of Morita theory.  In this section, we introduce this
organizational tool for those to whom it is unfamiliar.  For the
classical Morita theorem, we consider $\sM$, the bicategory of rings,
bimodules, and bimodule maps.  For Rickard's theorem, we consider
$DG_k$, the bicategory of DG $k$-algebras, $DG$-bimodules and their
maps.  Associated to this bicategory, we have a derived bicategory,
$\sD_k$.  We define these bicategories below, and in the remainder of
this section we discuss bicategories with a triangulated structure,
taking $\sD_k$ as a motivating example.  Precise and concise
definitions can be found in \cite{leinster1998bb}, while
\cite{lack20072cc} provides a more expanded guide.

\sbs{ Rings and Modules} The 0-cells of $\sM$ are rings, and for any
rings $A$ and $B$, $\sM(A,B)$ is the category of $(B,A)$-bimodules.
So a $(B,A)$-bimodule $_B M_A$ is a 1-cell $\cell{M}{A}{B.}$ The
2-cells between two 1-cells $\cell{M}{A}{B}$ and $\cell{N}{A}{B}$ are
the bimodule maps $f:\;\, _B M_A \to \,_B N_A$.  Given three 0-cells,
$A$, $B$, and $C$, and two 1-cells, $\cell{M}{A}{B}$ and
$\cell{L}{B}{C}$, the horizontal composite of $L$ with $M$ is written
$L \odot M : A \to C$.  Since $M$ is a $(B,A)$-bimodule, and $L$ is a
$(C,B)$-bimodule, $L \odot M$ is defined using the tensor product over
$B$; this produces a $(C,A)$-bimodule, as desired:
\[
L \odot M = L \otimes_B M.
\]

A bicategory has, for each 0-cell, $A$, a unit 1-cell $A \to A$
satisfying usual unit conditions.  We denote this 1-cell also by $A$;
in the case of rings, this is $A$ regarded as an $(A,A)$-bimodule.

\sbs{ Closed structure for $\sM$} A \emph{closed structure} for a
bicategory defines right adjoints for $\odot$.  For the bicategory
$\sM$ the right adjoints for $- \odot M$ and $M \odot -$ are well
known.  The right adjoint to $- \otimes_B M$ is $\Hom_A(\,M_A,-)$,
homomorphisms of right $A$-modules, while the right adjoint to $M
\otimes_A -$ is $\Hom_B(_BM,-)$, homomorphisms of left $B$-modules.
To extend these notions to more general bicategories, the adjoint to
$- \otimes_B M$ is called ``target-hom'', or ``right-hom'', and
denoted $M \rt -$.  The adjoint to $M \otimes_A -$ is called
``source-hom'', or ``left-hom'', and denoted $- \lt M$.  The
adjunctions are written as
\[\begin{array}{ccc}
  \sM(V \odot M,W) & \iso & \sM(V,M \rt W)\\
  \sM(M \odot T,U) & \iso & \sM(T,U \lt M)
\end{array}\]

The existence of these adjoints is a closed structure for a general
bicategory, and we will use $\rt$ and $\lt$ to denote the right-hom
and left-hom functors in general.  The orientation of the triangles is
intended to help the reader remember the source and target of the
1-cells $M \rt W$ and $U \lt M$.  Here, $W$ and $M$ have common
source, $A$, and if $C$ denotes the target of $W$, then $M \rt W$ is a
1-cell $B \to C$.  Likewise, $M$ and $U$ have common target, $B$, and
if $D$ denotes the source of $U$, then $U \lt M$ is a 1-cell $D \to
A$.  A technically complete description of closed structures can be
found in \cite{may2006pht}.

\sbs{ Differential Graded $k$-algebras} The bicategory $DG_k$ is
similar to $\sM$, but here the 0-cells are differential graded
$k$-algebras, the 1-cells are DG bimodules, and the 2-cells are maps
of DG bimodules.  Like $\sM$, $DG_k$ also has a closed structure.  For
two 1-cells with common source, $P$ and $Q$, the target-hom $P \rt Q$
denotes the differential graded hom over their common source, and
likewise $\lt$ denotes the differential graded hom over common
targets.

\sbs{ The derived bicategory, $\sD_k$} For each pair of DG
$k$-algebras, $A$ and $B$, there is a model structure for $DG_k(A,B)$
which is a direct generalization of the standard model structure for
chain complexes over a ring, and which \autoref{Model} describes in
more detail.  We use this model structure to understand and work with
the derived category of $(B,A)$-bimodules, which we denote by
$\sD_k(A,B)$. There is a canonical functor from $DG_k(A,B)$ to
$\sD_k(A,B)$ and where it adds clarity to our exposition we let $\ga:
DG_k(A,B) \to \sD_k(A,B)$ denote this functor.  Note that, for a
$k$-algebra $S$, $\sD_k(S,k) = \sD_k(S)$ is the usual derived category
of (right) $S$-modules.  The model structure on each 1-cell category
satisfies the pushout product condition for $\odot$-composition
(\autoref{pop}), so $DG_k$ is a model bicategory.  The derived tensor
and hom give a closed bicategory structure for the categories
$\sD_k(A,B)$, so we regard $\sD_k$ as the derived bicategory of
$DG_k$.

\sbs{ Duality in bicategories}\label{dltybicat} Throughout this
subsection we consider fixed 1-cells $X: B \to A$ and $Y: A \to B$ in
a closed bicategory $\sB$.

\begin{defn}[Dual pair]
  We say $(X,Y)$ is a dual pair, or `$X$ is left-dual to $Y$' (`$Y$ is
  right-dual to $X$'), or `$X$ is right-dualizable' (`$Y$ is
  left-dualizable') to mean that we have 2-cells

  \[
  \eta :A \to X \odot Y \text{\quad and \quad} \epz: Y \odot X \to B
  \]

  such that the following composites are the respective identity 2-cells.

  \[
  X \iso A \odot X \fto{\eta \odot \id} X \odot Y \odot X \fto{\id
    \odot \epz} X \odot B \iso X
  \]

  \[
  Y \iso Y \odot A \fto{\id \odot \eta} Y \odot X \odot Y \fto{\epz
    \odot \id} B \odot Y \iso Y
  \]
\end{defn}

\begin{defn}[Base and cobase for a dual pair]
  When $(X,Y)$ is a dual pair in a bicategory $\sB$, we term the
  source of $X$ (the target of $Y$) the \emph{base} of the dual pair,
  and we term the source of $Y$ (the target of $X$) the \emph{cobase}
  of the dual pair.  Thus, the evaluation map of the dual pair is a
  two-cell from $Y \odot X$ to the base 1-cell, and the coevaluation
  (unit) is a two-cell from the cobase 1-cell to $X \odot Y$.  
\end{defn}

\begin{defn}[Invertible pair]\label{invDefn}
  A dual pair $(X,Y)$ is called invertible if the maps $\eta$ and
  $\epz$ are isomorphisms.  Equivalently, the adjoint pairs described
  above are adjoint equivalences.
\end{defn}

Duality for monoidal categories has been studied at length, and
duality in a bicategorical context has been introduced in \cite[\S
16.4]{may2006pht}.  The definition of duality does not require $\sB$
to be closed, but we will make use of the following basic facts about
duality, some of which do require a closed structure on $\sB$.

\begin{prop}
  A 1-cell $X \elt \sB(A,B)$ is right-dualizable if and only if the coevaluation
  \[
  \nu: X \odot (X \rt A) \to X \rt X
  \]
  is an isomorphism.  Moreover, this is the case if and only if the
  map
  \[
  \nu_Z: X \odot (X \rt Z) \to X \rt (X \odot Z)
  \] 
  is an isomorphism for all 1-cells $Z$ with target $A$.
\end{prop}

\begin{prop}\label{dltyAdj} Let $(X,Y)$ be a dual pair in $\sB$, with $X: B \to A$ and $Y: A
  \to B$.
\begin{enumerate}
\item For any 0-cell $C$, we have two adjoint pairs of functors, with
  left adjoints written on top:
  \[\cmxymat{
    {\sB}(A,C) \ar[r]<.4ex>^-{- \odot X} & {\sB}(B,C) \ar[l]<.4ex>^-{- \odot Y}
  }\]
  \[\cmxymat{
    {\sB}(C,A) \ar[r]<.4ex>^-{Y \odot -} & {\sB}(C,B) \ar[l]<.4ex>^-{X \odot -}
  }\]
  The structure maps for the dual pair give the triangle identities
  necessary to show that the displayed functors are adjoint pairs.\\

\item If $\sB$ is closed, then $Y$ is canonically isomorphic to $X \rt
  B$, and for any 1-cell $Z: B \to D$, the natural map $Z \odot (X \rt
  B) \to X \rt Z$ is an isomorphism.
\end{enumerate}
\end{prop}

The right-dualizable 1-cells in the bicategory $\sM$ are the
finitely-generated projective bimodules. More precisely, they are
finitely-generated projective as right-modules over their source (the
base of the duality).  \autoref{retDlz} shows that the retracts of
\emph{finite cell} bimodules (\autoref{cellmod}) are right-dualizable
in $\sD_k$, and \autoref{retDlz-converse} shows that the converse is
also true.

\sbs{ Triangulated bicategories}\label{trbicat} We recall first the
definitions of localizing subcategory and generator for a triangulated
category, and then give a definition (\ref{trbicatDef}) of
triangulated bicategory suitable for our purposes.  In particular,
under this definition $\sD_k$ is a triangulated bicategory.

\begin{defn}[Localizing subcategory]\ \\
  If $\sT$ is a triangulated category with infinite coproducts, a
  \emph{localizing} subcategory, $\sS$, is a full triangulated
  subcategory of $\sT$ which is closed under coproducts from $\sT$.
\end{defn}

\begin{rmk}
  This is equivalent to the definition for arbitrary triangulated
  categories of \cite{hovey1999mc}, (which requires that a localizing
  subcategory be thick) because a triangulated subcategory
  automatically satisfies the 2-out-of-3 property and because in any
  triangulated category with countable coproducts, idempotents have
  splittings.  See \cite[1.5.2, 1.6.8, and 3.2.7]{neeman2001tc} for details.
\end{rmk}

\begin{defn}[Triangulated generator]\label{trGen}\ \\
  A set, $\sP$, of objects in $\sT$ (triangulated category with
  infinite coproducts, as above) is a set of \emph{triangulated
    generators} (or simply \emph{generators}) if the only localizing
  subcategory containing $\sP$ is $\sT$ itself.
\end{defn}

\begin{defn}[Triangulated bicategory {\cite[\S 16.7]{may2006pht}}]  
  \label{trbicatDef}\ \\
  A closed bicategory $\sB$ will be called a \emph{triangulated
    bicategory} if for each pair of 0-cells, $A$ and $B$, $\sB(A,B)$
  is a triangulated category with infinite coproducts, and if the
  suspension, $\SI$, is a pseudofunctor (\autoref{psF}) on $\sB$, and
  furthermore the local triangulations on $\sB$ are compatible as
  described in the following two axioms.
\begin{itemize}
\item[(TC1)] For a 1-cell $\cell{X}{A}{B}$, there is a
  natural isomorphism
  \[
  \al: X \odot \SI A \to \SI X
  \]
  such that the composite below is multiplication by $-1$.
  \[
  \SI^2 A = \SI(\SI A) \fto{\al^{-1}} \SI A \odot \SI A \fto{\kga} \SI
  A \odot \SI A \fto{\al} \SI(\SI A) = \SI^2 A
  \]

\item[(TC2)] For any 1-cell, $W$, the functors $W \odot -$, $- \odot
  W$, $W \rt -$, and $- \rt W$ are exact.
\end{itemize}
\end{defn}

If $\sB$ is a triangulated bicategory and $P$, $Q$ are 1-cells in
$\sB(A,B)$, we emphasize that $\sB$ is triangulated by writing the
abelian group of 2-cells $P \to Q$ as $\sB[P,Q]$ and by writing the
graded abelian group obtained by taking shifts of $Q$ as $\sB[P,Q]_*$.
To emphasize the source and target of $P$ and $Q$, we may also write
$\sB(A,B)[P,Q]_*$, as in \autoref{Ri} (where we write $\sD_k(S)$
instead of $\sD_k(S,k)$).

\section{Proof of \autoref{RiP1}}\label{proof}
In this section we prove \autoref{RiP1}, which generalizes one
direction of Rickard's theorem to the case of DG $k$-algebras.  We
work in the closed triangulated bicategory $\sD_k$, beginning with a
few general statements.

\begin{defn}[$\odot$-detecting 1-cells]\label{odD}\ \\
  In any locally additive bicategory, $\sB$, a 1-cell $\cell{W}{A}{B}$
  is called \emph{$\odot$-detecting} if triviality for any 1-cell
  $\cell{Z}{C}{A}$ is detected by triviality of the composite $W \odot
  Z$.  That is, $\cell{Z}{C}{A}$ is zero if and only if $W \odot Z
  = 0$.  A collection of 1-cells, $\sE$, in $\sB(A,B)$ is called
  \emph{jointly $\odot$-detecting} if the objects have this property
  jointly; that is, $Z = 0$ if and only if $W \odot Z = 0$ for all $W
  \elt \sE$.
\end{defn}

\begin{rmk}
  If $\sB$ is a monoidal additive category with monoidal product
  $\odot$, the unit object is $\odot$-detecting.  In arbitrary locally
  additive bicategories, if $A \ne B$ then $\sB(A,B)$ may not have a
  single object with this property, but in relevant examples the
  collection of all 1-cells, ob$\sB(A,B)$, does have this property
  jointly.  As a counter-point to this remark, we have the following
  lemma.
\end{rmk}

\begin{lem}\label{odDlem}
  Let $\sB$ be a triangulated bicategory, and let $\cell{P}{A}{B}$ be a
  generator for $\sB(A,B)$.  If the collection of all 1-cells,
  $\sB(A,B)$, is jointly $\odot$-detecting, then $P$ is
  $\odot$-detecting.
\end{lem}
\begin{proof}
  Given any 1-cell $\cell{Z}{C}{A}$ with $P \odot Z = 0$, let $\sS$ be
  the full subcategory of 1-cells, $\cell{W}{A}{B}$ for which $W \odot
  Z = 0$.  This is a localizing subcategory of $\sB(A,B)$, and by
  assumption $P \elt \sS$, so $\sS = \sB(A,B)$, and hence $Z = 0$.
\end{proof}

\begin{rmk}\label{odDrmk}
  Since the functors $P \odot -$ are exact, the property of $P \odot
  -$ detecting trivial objects is equivalent to $P \odot - $ detecting
  isomorphisms (meaning that a 2-cell $f$ is an isomorphism if and
  only if $P \odot f$ is so).
\end{rmk}

Now we turn to the proof.  Suppose $T$ is a chain complex of (right)
$S$-modules satisfying the dualizability and generator conditions of
\autoref{RiP1}, and let $E$ denote the DG $k$-algebra
$\Hom_S(T,T)$. One might call our first step `cobase extension', as we
describe how to extend the dualizable object $T$ to a dual pair with
base $S$ and cobase $E$.  The chain complex $T$ is a right DG-module
over $S$, and can be considered as a left module over the DG
$k$-algebra $E$.  We let $\wt{T}$ denote $T$ regarded as a 1-cell $S
\to E$, and let $T$ denote the 1-cell $S \to k$.  These 1-cells are
related by base change along the unit map (of DG $k$-algebras) $k \to
E$.  Restricting scalars on either the left or right of the DG
$k$-algebra $E$ gives rise to a dual pair $({_k E}, E_k)$, and the
1-cell $T : S \to k$ is recovered from $\wt{T} : S \to E$ as the
1-cell composite ${_kE} \odot \wt{T}$.  Moreover, $T \rt S$ is
recovered as the composite $(\wt{T} \rt S) \odot E_k$, and the
coevaluation map of $(T,T \rt S)$ is recovered from that of
$(\wt{T},\wt{T} \rt S)$.  In more common language, one might say that
$T \rt S$ is a right $E$-module, and the coevaluation map of $(T, T
\rt S)$ is a map of $E$-$E$ bimodules.

Because $T$ is (right-)dualizable in $DG_k(S,k)$, it follows that
$\wt{T}$ is (right-)dualizable in $DG_k(S,E)$.  \autoref{gaDlz}
below shows that, therefore, $\wt{T}$ is dualizable in
$\sD_k(S,E)$.  To finish the proof, we use this fact to show that,
because $T$ is a generator, $\wt{T}$ is invertible.  Then the
invertible pair $(\wt{T},\wt{T} \rt S)$ establishes an
equivalence of categories, as described in \autoref{invDefn}.

\begin{lem}\label{gaDlz}
  If $X$ is dualizable in $DG_k(A,B)$, then $\ga(X)$ is dualizable in
  $\sD_k(A,B)$.
\end{lem}
\begin{proof}
  The dualizable objects in $DG_k(A,B)$ are retracts of
  finitely-generated projective (right-)modules over $A
  \otimes_kB^{op}$, and hence cofibrant.  This is shown by a standard
  dual-basis-type argument, and the interested reader will find the
  details in \autoref{dlz-cofib}.  Because $X$ is cofibrant, the
  functor $X \odot -$ preserves weak equivalences.  This also is a
  standard result, and a proof can be found in
  {\cite[III.4.1]{kriz1995oam}}.

  Recall that $\ga$ denotes the canonical functor $DG_k(A,B) \to
  \sD_k(A,B)$.  Let $Q(X \rt A)$ be a cofibrant replacement for $X \rt
  A$.  Then $\ga(X) \odot \ga (X \rt A) = X \odot Q(X \rt A)$, and
  $\ga(X \odot (X \rt A)) = X \odot (X \rt A)$, and even though $\ga $
  is not a strong monoidal functor, we nevertheless have an
  isomorphism in $\sD_k(B,B)$: $\ga(X) \odot \ga(X \rt A) \fto{\hty}
  \ga(X \odot (X \rt A))$.  It is a formality now to check that the
  duality relations hold for $\ga(X)$ and $\ga(X \rt A)$, and
  therefore $\ga(X)$ is dualizable in $\sD_k(A,B)$.  For those who
  wish to see it, this formal argument is given explicitly for
  monoidal categories in \cite[III.1.9]{lewis1986esh}.
\end{proof}

Applying the lemma to $\wt{T}$ we have, for any $C$, the adjoint pair
of functors induced by a dual pair (\autoref{dltyAdj}) shown below.
Because $E = \wt{T} \rt \wt{T}$, the unit of this adjunction is an
isomorphism--the inverse to the coevaluation map.
\[\cmxymat{
  {\sD}_k(E,C) \ar@<.4ex>^-{- \odot \wt{T}}[r] & **[r]{\sD}_k(S,C). \ar@<.4ex>^-{- \odot (\wt{T} \rt S)}[l]
}\]

We finish the proof of \autoref{RiP1} by showing that the counit
$eval: (\wt{T} \rt S) \odot \wt{T} \to S$ is an isomorphism in
$\sD_k(S,S)$. Since $k$ is the ground ring for our bicategory $\sD_k$,
the 1-cells of $\sD_k(S,k)$ are jointly $\odot$-detecting
(\autoref{odD}), and so the generator condition of \autoref{RiP1}
means that $T$ itself is $\odot$-detecting (\autoref{odDlem}).  Thus,
evaluation $(\wt{T} \rt S) \odot \wt{T} \to S$ is an isomorphism in
$\sD_k(S,S)$ if and only if the map $({_kE} \odot \wt{T}) \odot
(\wt{T} \rt S) \odot \wt{T} \fto{1 \odot \text{eval}} {_kE} \odot
\wt{T}$ is so (\autoref{odDrmk}).  The duality of $\wt{T}$ and $\wt{T}
\rt S$ implies that the composite below is the identity and the first
map, induced by the unit of the adjunction, is an isomorphism so the
second must be also.
\[ {_kE} \odot \wt{T} \fto{\iso} {_kE} \odot \wt{T}
\odot (\wt{T} \rt S) \odot \wt{T} \fto{1 \odot
  \text{eval}} {_kE} \odot \wt{T}
\]
Hence the second map in this composite is an isomorphism, and so the
counit for the dual pair $(\wt{T},\wt{T} \rt S)$ is an
isomorphism in $\sD_k(S,S)$, giving an equivalence of triangulated
categories, as we wished to show.
\[\cmxymat{
  {\sD}_k(E,C) \ar@<1ex>^-{- \odot \wt{T}}[r] & **[r]{\sD}_k(S,C) \ar@<1ex>^-{- \odot (\wt{T} \rt S)}_-{\hty}[l]\\
}\]
This equivalence is suitably natural in $C$, making it a strong
transformation of the represented pseudofunctors $\sD_k(E,-)$ and
$\sD_k(S,-)$.  A more complete picture of strong transformations and
their connection to the Yoneda Lemma for bicategories is described in
\autoref{yoneda}.

\sbs{ Rickard's theorem for spectra}\label{RiSp} In this subsection we
prove a result analogous to \autoref{RiP1}, but working instead with a
commutative $\bS$-algebra, $k$, and the bicategory $\sS_k$ of
$k$-algebras and their bimodules.  We extend our previous notation to
let $\sD_k$ denote the bicategory of derived categories for spectra.
One major difference is that instead of working with dualizability on
the level of model categories, as we have for algebraic derived Morita
theory, we shift to the notion of dualizability in the bicategory of
derived categories.  The principles and general approach are the same,
but the details must be modified slightly.  For example, the `cobase
extension' step in the algebraic case is nearly transparent, but
requires a lemma in the context of spectra.  One can prove results
about spectra which are dualizable on the model-categorical level but,
unlike the DG case, it is difficult to find examples of such spectra.
With a relative abundance of spectra which are dualizable in the
derived bicategory, we shift our focus in that direction.  For the
remainder of this section, we use the term `dualizable' to mean
dualizable in the bicategory $\sD_k$.

\begin{prop}\label{htyDlzP1}
  Let $A$ be a $k$-algebra, and let $T$ be a fibrant and cofibrant
  $A$-module, with endomorphism $k$-algebra $E = F_A(T,T)$.  If $T$
  has the following two properties, then $\sD_k(A)$ and $\sD_k(E)$ are
  equivalent categories.
  \begin{enumerate}
    \renewcommand{\theenumi}{\textit{\roman{enumi}}}
  \item $T$ is (right-)dualizable as an $A$-module.
  \item $T$ generates the triangulated category $\sD_k(A)$.
  \end{enumerate}
\end{prop}

As with the algebraic version, our proof proceeds in two parts.  First
(`cobase extension') we show that a dual pair in $\sD_k$ between two
$k$-algebras can be extended to the endomorphism algebra of the left
dual, and that in so doing we produce a new dual pair whose unit is an
isomorphism.  This rough description is made precise in the statement
of \autoref{htyDlzCounit}, after introducing notation for the
restriction of scalars functors.  In the second part of our proof we
use the unit isomorphism of this new dual pair, together with a
generating condition, to detect that the evaluation map is also an
isomorphism (in $\sD_k$).  Hence the new dual pair is an invertible
pair, giving an equivalence of categories.

\begin{notn}
  Given a map of $k$-algebras $\io :B \to E$, we have two
  restriction-of-scalars functors: one for restriction of left
  modules, and another for restriction of right modules.  For any
  $k$-algebra $A$, We let $\io_L^* : \sS_k(A, E) \to \sS_k(A,B)$
  denote restriction on the left (target), and $\io_R^* : \sS_k(E,A)
  \to \sS_k(B,A)$ denote restriction on the right (source).  Both
  functors create weak-equivalences and fibrations.
\end{notn}

% do not change from 'lem' without modifying references
\begin{lem}\label{htyDlzCounit}
  Let $A$ and $B$ be $k$-algebras, and let $T$ be fibrant and
  cofibrant in $\sS_k(A,B)$, with endomorphism $k$-algebra $E =
  F_A(T,T)$.  If $T$ is (right-)dualizable in $\sD_k$, then there is a
  homotopy dual pair $(\wt{T},\wt{D})$ with base $A$ and cobase $E$
  whose unit is an isomorphism.  This dual pair extends $T$ in the
  sense that $\iota_L^*\wt{T} \hty T$, where $\iota : B \to E$ is the
  $k$-algebra map adjoint to the action of $B$ on $T$.
\end{lem}

\begin{proof}
  Because $T$ is cofibrant and fibrant in $\sS_k(A,B)$, no
  replacements are necessary and $E$ is the derived endomorphism
  monoid of $T$.  The unit map $k \to E$ is obtained as the composite
  of algebra maps $k \to B \to E$.  Let $\wt{T}$ be a cofibrant
  replacement for $T$ in $\sS_k(A,E)$.  Recall that $T$ is cofibrant
  in $\sS_k(A,B)$, and hence has the LLP with respect to acyclic
  fibrations.  We construct $\wt{T}$ by the usual factorization of the
  map from the initial object, and the forgetful functor $\iota_L^*$
  creates weak equivalences and fibrations, so the lifting property
  for $T$ gives a weak equivalence $T \fto{\hty} \iota_L^*\wt{T}$.

  The canonical dual of $T$ is $F_A(T,A) = T \rt A \elt \sS_k(B,A)$,
  and we let $D$ denote a cofibrant replacement for $F_A(T,A)$ in
  $\sS_k(B,A)$, so that we have a weak equivalence $D \fto{\hty}
  F_A(T,A)$.  The canonical dual of $T$ has a right-action of the
  endomorphism $k$-algebra, $E$, and we let $\wt{D}$ be a cofibrant
  replacement for $F_A(T,A)$ in $\sS_k(E,A)$, constructed again by the
  usual factorization.  Since the forgetful functor $\io_R^*$ creates
  weak equivalences and fibrations, we have an acyclic fibration
  $\io_R^* \wt{D} \fto{\hty} F_A(T,A)$ in $\sS_k(B,A)$.  Because $D$ is
  cofibrant, the weak equivalence $D \fto{\hty} F_A(T,A)$ lifts with
  respect to acyclic fibrations and hence we have a weak equivalence
  $D \fto{\hty} \iota_R^*\wt{D}$.

  Now we show that $(\wt{T},\wt{D})$ is a dual pair in $\sD_k$.  The
  weak equivalences $\wt{T} \to T$ and $\wt{D} \to F_A(T,A)$ in
  $\sS_k(A,E)$ and $\sS_k(E,A)$, respectively, give maps 
  \[
  \wt{T} \odot \wt{D} \to T \odot F_A(T,A) \to E \text{ and } \wt{D}
  \odot \wt{T} \to F_A(T,A) \odot T \to A
  \]
  in $\sS_k(E,E)$ and $\sS_k(A,A)$, respectively.  Moreover, the first
  map is an isomorphism in $\sD_k(E,E)$ because its image under
  $\iota^*_L\iota^*_R$ is a composite of two isomorphisms in
  $\sD_k(B,B)$: 
  \[
  \iota^*_L\wt{T} \odot \iota^*_R\wt{D} \iso T \odot D \iso
  \iota^*_L\iota^*_RE.
  \]
  The inverse to this map gives the unit for the dual pair, and the
  duality diagrams commute because the corresponding diagrams for $T$
  and $F_A(T,A)$ do.  Hence the functors $- \odot \wt{T}$ and $- \odot
  \wt{D}$ induce an adjunction
  \[\cmxymat{
    {\sD}_k(A,C) \ar@<.4ex>[r]^-{- \odot \wt{T}} & {\sD}_k(E,C) \ar@<.4ex>[l]^-{- \odot \wt{D}}\\
  }\]
  and the unit of this adjunction is an isomorphism.
\end{proof}

% do not change from 'lem' without modifying next paragraph
\begin{lem}\label{htyDlzODD}
  Let $T$, $\wt{T}$, $\wt{D}$ be as in \autoref{htyDlzCounit}, with $B = k$.  If $T$ generates $\sD_k(A,k)$, then $\wt{T}$ is $\odot$-detecting.
\end{lem}
\begin{proof}
  As in the algebraic case, this follows because $k$ is the ground
  object, and hence the collection of all 1-cells is jointly
  $\odot$-detecting (\autoref{odD}).  The generator condition
  therefore implies that $T$ itself is $\odot$-detecting
  (\autoref{odDlem}).  Now $\iota_L^*$ creates weak equivalences, and
  $\iota_L^*(-) = \iota_L^*(E) \odot -$, so $\wt{T}$ is also
  $\odot$-detecting.
\end{proof}

Using Lemmas \ref{htyDlzCounit} and \ref{htyDlzODD}, we finish the
proof of \autoref{htyDlzP1} as in the algebraic case.  Both the
composite and the first map displayed below are isomorphisms, and
hence the second map is also an isomorphism.  But the second map is
$\wt{T} \odot -$ applied to the counit, and since $\wt{T}$ is
$\odot$-detecting, the counit of the dual pair is therefore an
isomorphism in $\sD_k(A,A)$.
\[
E \odot \wt{T} \to \wt{T} \odot \wt{D} \odot \wt{T} \to \wt{T} \odot A
\]

\section{The Bicategorical Yoneda Lemma}\label{yoneda}
This section describes the Yoneda Lemma for bicategories.  Following
\cite{street1980fb}, we avoid giving the detailed definitions, and
instead give some general description followed by examples, which will
be our main interest.  As in \autoref{context}, we suggest
\cite{lack20072cc} or \cite{leinster1998bb} for further background.

\sbs { Pseudofunctors}\label{psF} If $\sA$ and $\sB$ are bicategories,
a pseudofunctor $\CMcal{P}:\sA \to \sB$ (also called a {\em morphism}) is the
bicategorical version of a functor.  It is a function on 0-cells and
for each pair of 0-cells a functor
\[\sA(A,B) \fto{\CMcal{P}_{AB}} \sB(\CMcal{P}A,\CMcal{P}B).\]
These functors are compatible with $\odot$-composition in that there
are 2-cell isomorphisms
\[\CMcal{P}_{BC}X' \odot \CMcal{P}_{AB}X \fto{\iso} \CMcal{P}_{AC}(X' \odot X)\]
satisfying the natural associativity and unit compatibility conditions.

Our focus is on the represented pseudofunctors.  These are a
bicategorical version of represented functors for categories, and they
take values in the bicategory $Cat$.  In this bicategory, the 0-cells
are categories, the 1-cells are functors, and the 2-cells are natural
transformations of functors.  For any bicategory $\sB$ with 0-cell
$A$, we have the represented pseudofunctor $\sB(A,-) : \sB \to Cat$.
For a 0-cell $E \elt \sB$, this pseudofunctor gives a category,
$\sB(A,E)$.  For a 1-cell $\cell{M}{E}{E'}$, we have the functor $M
\odot - : \sB(A,E) \to \sB(A,E')$ , and 2-cells $M \to M'$ give
natural transformations of such functors.  The compatibility
isomorphisms which make $\sB(A,-)$ a pseudofunctor are precisely the
associativity isomorphisms $(M_2 \odot (M_1 \odot -)) \iso (M_2 \odot
M_1) \odot -$.

In this context, our introductory remark `Morita theory is about
\emph{bi}modules' can be rephrased as the comment that Morita theory
is about represented pseudofunctors.  Our remark near the end of
\autoref{outline} that `standard derived equivalences preserve
bimodule structure' can be understood as the observation that standard
derived equivalences are transformations of represented
pseudofunctors.

\sbs{ (Strong) transformations} A transformation is a kind of
bicategorical natural transformation of functors.  A transformation of
two represented pseudofunctors, $\sB(B,-)$ and $\sB(A,-)$ is given by
\begin{enumerate}
\item A family of functors $F_C : \sB(B,C) \to \sB(A,C)$.  These are
  the \emph{components} of $F$.
\item For each 1-cell $C \fto{K} C'$, a natural transformation which,
  for 1-cells $X \elt \sB(B,C)$, has component 2-cells
  \[
  K \odot F_C(X) \to F_{C'}(K \odot X)
  \]
  natural in $K$ and $X$, with standard associativity and unit
  compatibilities; namely that the following diagrams commute, with
  $K$ and $X$ as above, and $L \elt \sB(C', C'')$.
  \[\cmxymat{
    L \odot K \odot F_C(X) \ar[dr] \ar[r] & L \odot F_{C'} (K \odot X) \ar[d] \\
    & F_{C''}(L \odot K \odot X)
  }
  \hspace{4pc}
  \cmxymat{
    C \odot F_C(X) \ar[d]^-{\iso} \ar[r] & F_C(C \odot X) \ar[dl]^-{\iso}\\
    F_C(X) &
  }\]
 
\end{enumerate}

\begin{notation}
  In the following, we will frequently drop the subscripts on the
  components of our transformations since they may always be
  determined from context and they tend to make the text less
  readable.
\end{notation}

For developing Morita theory, our interest will be in {\em strong}
transformations; these are transformations for which the component
2-cells shown above are natural isomorphisms.  Restricting attention
to strong transformations is equivalent to restricting to
transformations which have object-wise adjoints (\autoref{odotRep}).
Since the equivalences Morita theory seeks to understand are, in
particular, adjoint pairs, this restriction of scope is necessary.
Similar ideas are considered for distributors in
\cite{fisherpalmquist1975mce} and \cite{borceux2002ac} and for
bialgebroids in \cite{szlachanyi2004mme}.  

The appropriate morphisms of transformations are called {\em
  modifications}, but we will not make any explicit reference
to them beyond the following definition.

\begin{defn}
  For two bicategories $\sA$ and $\sB$, $\kPs_s[\sA, \sB]$ denotes the
  bicategory whose 0-cells are pseudofunctors $\sA \to \sB$, 1-cells
  are strong transformations, and 2-cells are modifications.
\end{defn}

\begin{lem}[Yoneda \cite{street1980fb}]\label{YL}
  For a pseudofunctor of bicategories $\CMcal{P} : \sA \to Cat$, evaluation
  at the unit 1-cell for each 0-cell, $A$, of $\sA$ provides the
  components for an equivalence of categories
  \[
  \kPs_s[\sA,Cat](\sA(A,-),\CMcal{P}) \fto{\hty} \CMcal{P}A.
  \]
\end{lem}

\begin{cor}[Morita II]\label{MoII}
\[\kPs_s[\sA,Cat](\sA(A,-),\sA(B,-)) \fto{\hty} \sA(B,A)\]
That is, strong transformations $\sA(A,-) \to \sA(B,-)$ are given
(precisely) by $\odot$-composition with a 1-cell $B \to A$.  In
particular, strong transformations which induce equivalences $\sA(A,C)
\simeq \sA(B,C)$ for all 0-cells $C$ are given by invertible 1-cells
$B \to A$.
\end{cor}

The essential point of the proof, as in the 1-categorical case, is the
observation that for a strong transformation, $S$, and a 1-cell $\cell{Z}{A}{C}$,
\[
S_C(Z) \iso S_C(Z \odot A) \iso Z \odot S_A(A)
\]
so that, for any $C$, the functor $S_C$ is determined by $S_A(A)$, an
object in the category $\sA(B,A)$.  Natural transformations of these
functors are determined by morphisms in $\sA(B,A)$.

This equivalence can be read with various emphases, yielding various
interpretations.  One possible interpretation would take strong
transformations or strong equivalences as objects of interest and take
the equivalence as a characterization of these objects--they can be
only those transformations arising as $\odot$-composition with a
1-cell.  A complementary interpretation takes the transformations
given by $\odot$-composition as the basic objects of interest, as in
the case of the standard derived equivalences for derived categories
of DG $k$-algebras.  From this point of view, the equivalence is an
assurance that functors arising in this way are no less, and no more,
than the strong transformations.

In the presence of a closed structure for our bicategory, we have a
further interpretation.  Functors given by $\odot$-composition with a
1-cell are naturally enriched over an ambient closed structure; in the
following proposition, we formalize what is meant by a family of
functors enriched over the internal hom, and show that such functors
are necessarily the family of components of a transformation.

\begin{prop}\label{rtF}
  Let $A$ and $B$ be 0-cells of a closed bicategory $\sB$, and let $F$
  be a family of functors
  \[F_C : \sB(A, C) \to \sB(B, C)\] 
  for 0-cells $C$.  The following are equivalent.
\begin{enumerate}
\item \label{ect1} 
  For any $C$, and any 1-cells $T \elt \sB(A,C_1)$, $U \elt
  \sB(A,C_2)$, $V \elt \sB(A,C_3)$ there are 2-cells
  \[T \rt U \to FT \rt FU \text{ in } \sB(C_1,C_2)\] and
  \[U \rt V \to FU \rt FV \text{ in } \sB(C_2,C_3).\] 
  These 2-cells are natural in $T$, $U$, and $V$, preserve units, and
  commute with composition, in the sense described by the following.
\[\begin{array}{cc}
\cmxymat{
(U \rt V) \odot (T \rt U) \ar[r]^-{\text{comp}} \ar[d] & T \rt V \ar[d]\\
(FU \rt FV) \odot (FT \rt FU) \ar[r]^-{\text{comp}} & FT \rt FV
}
&
\cmxymat{
T \rt T \ar[r] & FT \rt FT \\
C_1 \ar[u]^-{} \ar[ur]_-{\text{adj. to unit}}
}
\end{array}\]

\item \label{ect2} The family $F$ is the family of components for a
  transformation of represented pseudofunctors.  That is, for any
  1-cells $X \elt \sB(A,C_1)$ and $K \elt \sB(C_1,C_2)$, there are 2-cells
  \[K \odot F(X) \to F(K \odot X) \text{ in } \sB(B,C_2).\] 
  These 2-cells are natural in $K$ and $X$, and associative and
  unital.
\end{enumerate}
\end{prop}
\begin{proof}
  Given maps as in (\ref{ect1}), and 1-cells $K$ and $X$ as in
  (\ref{ect2}), we describe the structure 2-cell
  \[
  K \odot F(X) \to F(K \odot X)
  \]
  as the following composite:
  \[
  K \odot F(X) \fto{\text{adj. to } \id_{K \odot X}} \left[ X \rt (K
    \odot X) \right] \odot F(X) \fto{F} \left[ F(X) \rt F(K \odot X)
  \right] \odot F(X) \fto{\text{eval}} F(K \odot X).
  \]
  Using the definition of the map, associativity for the structure
  2-cell is reduced to the given commutativity with composition.
  Unitality follows from the unit condition above.

  The situation is exactly reversed for the converse.  Given maps as
  in (\ref{ect2}) and 1-cells $T$ and $U$ as in (\ref{ect1}), we
  describe
  \[
  T \rt U \fto{F} FT \rt FU
  \]
  as adjoint to the map
  \[
  (T \rt U) \odot FT \to F( (T \rt U) \odot T) \fto{F(\text{eval})}
  FU.\qedhere
  \]
\end{proof}

\autoref{rtF} shows that a transformation of represented
pseudofunctors can be interpreted as what one might call ``a natural
family of enriched functors''.  The following lemma is proved
similarly, and gives a specialized interpretation for the strong
transformations: families of left adjoints.
\begin{lem}\label{odotRep}
  For $F$ as above, the maps $K \odot F(X) \to F(K \odot X)$ are
  isomorphisms for all $K$ and $X$ if and only if $F$ has a family of
  right adjoints,
  \[\sB(A,C) \ot \sB(B,C): G_C\]
  and these adjoints have maps $K' \odot G(X') \to G(K' \odot X')$
  making $G$ into a transformation of represented pseudofunctors
  $\sB(B,-) \to \sB(A,-)$.  In other words, a transformation $F$ is a
  strong transformation if and only if it has a right adjoint
  transformation.
\end{lem}

\section{Practical Interpretation}\label{practical}
In this section, we return our focus to Morita theory.  The question
of when a derived equivalence is a \emph{standard} derived equivalence
is raised, but not answered, by Rickard's work.  \autoref{DSeg} shows
that a short answer to this question is ``not always'', and more
subtle answers have been explored in the literature of derived Morita
theory (see \cite{konig1998deg}, for example).  Our inspection of the
Yoneda Lemma yields a reinterpretation of this question, and another
approach to determining when one might give an affirmative answer; we
discuss this in \autoref{components}.  In \autoref{enrichment}, we
turn to the ambient enrichments which are present in both classical
and derived Morita theory.  We again use the ideas of the previous
section, this time to emphasize the relevance of enrichments to Morita
theory.  We follow this explanation with some examples, illustrating
how one might apply these interpretations in practice.

\sbs{ Components of (strong) transformations}\label{components} In
this subsection we let $\sB$ denote a closed bicategory, with the
example $\sB = \sD_k$ as our primary motivation.  Let $A$, $B$, and
$I$ be three fixed 0-cells of $\sB$; in our motivating example, we
take $I = k$.  Furthermore, suppose $f$ is simply a functor $\sB(A,I)
\to \sB(B,I)$.  Recalling \autoref{MoII} (Morita II), the question of
when $f$ is a \emph{standard} functor is equivalent to the question of
whether $f$ is a component of a transformation between pseudofunctors,
that is, whether there is a transformation $F: \sB(A,-) \to \sB(B,-)$
with $F_I = f$.  In particular, we seek to understand the case when
$f$ is an equivalence, and characterize when $f$ is a component of a
strong transformation.  As \autoref{odotRep} points out, $f$ being a
component of a strong transformation implies that its adjoint is
itself a component of a transformation.

In this situation, let $\zcb$ denote the full sub-bicategory of $\sB$
whose 0-cells are $A$ and $I$.  There are four 1-cell categories in
$\zcb$; these are the categories of 1-cells in $\sB$ from $A$ to $A$,
from $A$ to $I$, from $I$ to $A$, and from $I$ to $I$.  That is,
$\zcb(x,y) = \sB(x,y)$ for $x,y \elt \{A,I\}$.  We have the
represented pseudofunctor $\zcb(A,-) : \zcb \to Cat$, and we have also
the pseudofunctor $\sB(B,-): \zcb \to Cat$.  Since $B$ is not a 0-cell
of $\zcb$, $\sB(B,-)$ is un-represented, but it is nevertheless a
pseudofunctor on $\zcb$ and the Yoneda Lemma applies to describe
strong transformations $\zcb(A,-) \to \sB(B,-)$.  Using the Yoneda
lemma twice, we have the following two equivalences:
\[
\kPs_s[\zcb, Cat](\zcb(A,-),\sB(B,-)] \fto{\hty} \sB(B,A) \fot{\hty}
\kPs_s[\sB, Cat](\sB(A,-), \sB(B,-)).
\]

These two equivalences say in bicategorical language what is apparent
to one who considers the proof of the Yoneda Lemma: that strong
transformations are determined by their values on the unit 1-cell.
One could make this clearer by restricting $\zcb$ further to a single
0-cell, $A$, since for the equivalence above $I$ is irrelevant.  We
have chosen to include $I$ so that the equivalences above provide a
proof for the following.

\begin{cor}\label{Ycor}
  A functor $f : \sB(A,I) \to \sB(B,I)$ is a component of a strong
  transformation $\sB(A,-) \to \sB(B,-)$ if and only if it is a
  component of a strong transformation $F: \zcb(A,-) \to \sB(B,-)$.  That
  is, $f$ is a component of a strong transformation if and only if
  there is a functor $f' : \sB(A,A) \to \sB(B,A)$ and natural 2-cell
  isomorphisms $K \odot f'(X) \fto{\iso} f(K \odot X)$, with the
  apparent associativity requirement, for any 1-cells $X : A \to A$
  and $K: A \to I$.  In this case, the strong transformation $F$ is
  determined by its two components, $F_I = f$ and $F_A = f'$.
\end{cor}

Dropping the condition that the transformation be strong, we can use
\autoref{rtF} in the preceding context to achieve a description in
terms of the internal hom.

\begin{cor}\label{Ecor}
  A functor $f: \sB(A,I) \to \sB(B,I)$ is a component of a
  transformation $\sB(A,-) \to \sB(B,-)$ if and only if there is a
  functor $f': \sB(A,A) \to \sB(B,A)$ and, for 1-cells $T,U \elt
  \sB(A,I)$ and $T',U' \elt \sB(A,A)$, there are 2-cells
  \begin{center}
  \begin{tabular}{ll}
    $T \rt U \to fT \rt fU$ & in $\sB(I,I)$\\
    $T' \rt U'  \to f'T' \rt f'U'$ & in $\sB(A,A)$\\
    $T \rt U' \to fT \rt f'U'$ & in $\sB(I,A)$\\
    $T' \rt U \to f'T' \rt fU$ & in $\sB(A,I)$\\
  \end{tabular}
  \end{center}
  subject to the compatibility with composition and units described in \autoref{rtF}(\ref{ect1}).
\end{cor}

\begin{rmk}\label{Ermk}
  In this context, the compatibility means precisely that all possible
  diagrams of the form below commute:
\[
\cmxymat{
  (U \rt V) \odot (T \rt U) \ar[d] \ar[r] & T \rt V \ar[d]\\
  (FU \rt FV) \odot (FT \rt FU) \ar[r] & FT \rt FV } \hspace{2pc}
\cmxymat{
  T \rt T \ar[r] & FT \rt FT \\
  unit \ar[u] \ar[ur]& }
\] 
Where each of $T$, $U$, and $V$ is taken to be either in $\sB(A,A)$ or
$\sB(A,I)$, `$unit$' is taken to be either $I$, or $A$, and $F$ is
taken to be either $f$ or $f'$, as appropriate.  Note that if all
three of the objects, $T$, $U$, $V$ are taken from the same category,
this condition is precisely the condition which makes $f$ and $f'$
enriched over $\rt$.  The first diagram says that the enrichment must
commute with the composition pairing, and the second diagram says that
the enrichment must preserve units.
\end{rmk}

We now turn to some more practical interpretations of
\autoref{yoneda}.  Combining these results with our previous
interpretation of Morita theory enables us to give a description of
the concepts at work in both classical and derived Morita theory.  We
follow this description with examples of \autoref{Ecor}, showing that,
at least in algebraic contexts, the four-part necessary condition can
be verified formally.

\sbs{ Enriched equivalences}\label{enrichment} \autoref{rtF} shows that the standard
Morita equivalences in a closed bicategory must be equivalences which
are enriched over the internal hom, and likewise that families of
equivalences which are enriched over the internal hom fit together to
form standard Morita equivalences.  As noted in
\autoref{enrichmentRmk}, left-adjoint functors (e.g.  equivalences)
between abelian categories are automatically enriched in abelian
groups, and this may be one reason that enrichment has been
under-appreciated in these contexts.  The topological Morita theorem
of Schwede and Shipley \cite{schwede2003smc} addresses \emph{spectral}
Quillen equivalences--Quillen equivalences enriched in spectra.
Likewise, To{\"e}n \cite{toen2007htd} works with DG categories, the
morphisms of which are enriched functors.  \autoref{rtF} shows that
focusing on enriched equivalences is inevitable.

In these papers the authors also address the important question of
what model-theoretic assumptions could be verified in practice and
would guarantee standard Morita equivalences on the derived level.
One possible lesson taught by \autoref{DSeg} is that certainly some
assumptions are necessary in general.  The results above, however, are
independent of model theory, applying to any closed bicategory.  They
indicate that Quillen equivalences which induce \emph{enriched}
transformations on the derived level are necessarily the appropriate
equivalences for the development of Morita theory.  This perspective
can offer an explanation for the results of \cite{dugger2007ted} in
particular.  There, and in related work, the technical notion of
\emph{additive} model category is introduced, and it is shown that
Quillen equivalences between additive model categories are necessarily
additive functors, just as in the classical situation.  The result,
therefore, is that a zig-zag of Quillen equivalences for which each
intermediate model category is \emph{additive} provides a well-behaved
notion of Morita equivalence for additive model categories.  Our
perspective would suggest that this can be extended to more general
enriched model categories, with the appropriate notion of Morita
equivalence in those settings being enriched Quillen functors. From
this point of view, \autoref{DSeg} is an expected example, and others
like it will be expected in applications for which Quillen functors
are not necessarily enriched.

\autoref{Ecor} shows that the property of enrichment may be identified
by considering only a specific special case, arising through our
restriction from the bicategory $\sB$ to the full sub-bicategory,
$\zcb$ generated by two 0-cells, $A$ and $I$.  In algebraic examples,
this four-part condition can be simplified even further.  We
demonstrate this by recalling the following two classical results,
with their proofs for reference.  They show that, in the case of
classical Morita theory, the condition in \autoref{Ecor} is automatic.

\begin{thm}[Morita II {\cite[7.18.26]{lam1999lmr}}]\label{lam}
  Let $R$ and $S$ be two rings, and let
  \[\cmxymat{
  f : \sM(R,\bZ) \ar[r]<.4ex> & {\sM}(S,\bZ): g \ar[l]<.4ex> 
  }\]
  be an equivalence between the categories of right $R$-modules and
  right $S$-modules.  Let $Q = f(R)$ and let $P = g(S)$.  Then there
  are natural bimodule structures making $P \elt \sM(R,S)$ and $Q \elt
  \sM(S,R)$.  Using these bimodule structures, there are natural
  isomorphisms of functors
  \[
  f \iso - \odot Q \text{\quad and \quad} g \iso - \odot P.
  \]
\end{thm}
\begin{proof}
  The bimodule structures are recognized by the ring isomorphisms
  \[
  R \iso \Hom_R(R,R) \iso \Hom_S(fR,fR) \text{\quad and \quad} S \iso
  \Hom_S(S,S) \iso \Hom_R(gS,gS).
  \]
  The identification of $f$ is obtained by the following computation,
  using the fact that $g$ is an adjoint for $f$ and that $P$ is
  dualizable; the identification of $g$ is similar.  For $M \elt
  \sM(R, \bZ)$,
  \[
  f(M) \iso \Hom_S(S,fM) \iso \Hom_R(gS,M) \iso M \otimes_R
  \Hom_R(P,R) \iso M \otimes_R Q = M \odot Q.
  \]
\end{proof}

\begin{rmk}
  The proof here implicitly defines the functor $f'$ as the composite
  of the forgetful functor from $R$-$R$-bimodules to right $R$-modules
  with the functor $f$.  The computation above shows that the image of
  this composite lies in the subcategory of $S$-$R$-bimodules.
\end{rmk}

To relate the previous result to the following one, recall that
equivalences of abelian categories are, in particular, exact and
coproduct-preserving.
\begin{thm}[Watts \cite{watts1960ics}]\label{watts}
  Let $R$ and $S$ be rings, and let $f: \sM(R,\bZ) \to \sM(S,\bZ)$ be
  a functor from the category of right $R$-modules to the category of
  right $S$-modules.  If $f$ is right-exact and preserves direct sums,
  then there is a bimodule $C \elt \sM(S,R)$ and a natural isomorphism
  $f \iso - \odot C$.
\end{thm}
\begin{proof}
  Since $f$ preserves direct sums, it is automatically enriched over
  $\Hom_R$.  If $T'$ is an $(S,R)$-bimodule, we observe that $f({_{\bZ}T}')$ has
  a natural $S$-module structure given by the map of abelian groups
  \[
  \Hom_S({_\bZ S},{_\bZ S}) \to \Hom_R({_\bZ T'},{_\bZ T'}) \to
  \Hom_S(f({_\bZ T'}),f({_\bZ T'}))
  \]
  and we define $f'(T')$ to be the $(S,S)$-bimodule whose underlying
  $(\bZ,S)$ bimodule is $f({_\bZ T'})$.  The categories of bimodules
  $\sM(R,R)$ and $\sM(S,R)$ are defined to be the subcategories of
  left $R$-modules in $\sM(R,\bZ)$ and $\sM(S,\bZ)$, respectively, and
  hence the compatibility conditions of \autoref{Ecor} follow
  formally.  This shows that $f$ and $f'$ are components of a
  transformation $\sM(R,-) \to \sM(S,-)$.  The theorem is proven once
  we show that this is a strong transformation.  By \autoref{odotRep},
  it suffices to show that this transformation has a right-adjoint
  transformation.  This also follows formally, by the special adjoint
  functor theorem: $f$ is coproduct-preserving, and right-exact, and
  hence has a right-adjoint; $f'$ likewise has a right adjoint, and
  these form an adjoint transformation.
\end{proof}

In derived contexts, the condition of \autoref{Ecor} is no longer
automatic, but we can still apply the result to obtain the following
explicit description.

\begin{prop}\label{StdDerEquivs}
  Let $k$ be a commutative ring, let $A$ be a DG $k$-algebra and let
  $f : \sD_k(A) \to \sD_k(\End_k(A))$ be an equivalence of triangulated
  categories.  Then $f$ is a standard derived equivalence if and only
  if we have the following:
  \begin{enumerate}
    \renewcommand{\theenumi}{\textit{\roman{enumi}}}
  \item The equivalence given by $f$ is an enriched equivalence.
  \item There is an enriched equivalence $f' : \sD_k(A,A) \to \sD_k(\End_k(A), A)$.
  \item The two equivalences, $f$ and $f'$ are compatible in the
    following sense: If $T', U' \elt \sD_k(A,A)$ and $T,U \elt
    \sD_k(A) = \sD_k(A,k)$, then there are natural maps
    \begin{center}
    \begin{tabular}{ll}
    $\Ext_A(T,U') \to \Ext_{\End_k(A)}(fT,f'U')$ & in $\sD_k(k,A)$\\
    $\Ext_A(T',U) \to \Ext_{\End_k(A)}(f'T',fU)$ & in $\sD_k(A,k)$
    \end{tabular}
    \end{center}
    which commute with the pairing induced by composition.  (That is,
    the squares in \autoref{Ermk} commute.)
  \end{enumerate}
\end{prop}

\section{Model Structure for DG Algebras}\label{Model}
We begin with some definitions and reminders from \cite{kriz1995oam},
and give a model structure for the category of DG-modules over a DG
algebra.  We make use of the homotopy extension and lifting property
(HELP) to streamline the model-theoretic arguments, and emphasize the
analogy with topology.  When the DG algebra is concentrated in degree
0 (a ring), this is the standard model structure for chain complexes
over the ring (\autoref{nothingnew}).  Let $k$ be a commutative ring
and $A$ a fixed DG $k$-algebra.  We let $S^n = A \otimes_k (k[n])$,
where $k[n]$ is a free DG $k$-module on a single generator in degree
$n$, so $S^n$ is a free $A$-module on a single generator in
degree $n$.  We let $D^n$ be a free $A$-module with one
generator in degree $n$ and one in degree $n-1$; the differential on
$D^n$ takes the generator in degree $n$ to that in degree $n-1$.
Finally, we let $I$ denote a free $A$-module which has one
generator, $\lng I \rng$, in degree 1, and two generators, $\lng 0
\rng$ and $\lng 1 \rng$ in degree 0; on generators, the differential
in $I$ is given by $\lng I \rng \mapsto \lng 0 \rng - \lng 1 \rng$.
We let $\otimes$ denote $\otimes_A$, and for any $A$-module $M$ we let
$i_0$ and $i_1$ denote the inclusions $M \to M \otimes I$
corresponding to $M \otimes \lng 0 \rng$ and $M \otimes \lng 1 \rng$,
respectively.

\begin{defn}[Relative cell module]\label{cellmod}
  A map of $A$-modules $C_0 \to C$ is called a relative cell
  $A$-module if $C$ is the colimit of a sequence of maps $C_r \to
  C_{r+1}$, with each map obtained as a pushout
  \[\cmxymat{
    {\bigoplus_{q_i}} S^{q_i} \ar[r] \ar[d] & C_{r} \ar[d] \\
    {\bigoplus_{q_i}} D^{q_i + 1} \ar[r] & C_{r+1} \\
  }\]
  The maps $S^{q_i} \to C_{r}$ above are called the \emph{attaching
    maps} for $C_{r}$.  If $0 \to C$ is a relative cell $A$-module,
  $C$ is called a cell $A$-module.  If there are only finitely many
  cells, then $C$ is called a finite cell $A$-module.
\end{defn}

This is generalized by the following definition.

\begin{defn}[{\cite[4.5.1]{may2006pht}}]\label{MSdef}
Let $\sI$ be a set of maps in a category $\sC$ with coproducts $\oplus$.
\begin{enumerate}
\renewcommand{\theenumi}{\textit{\alph{enumi}}}
\item A \emph{relative $\sI$-cell module} is a map $C_{0} \to C$, with
  $C$ obtained as a colimit of maps, $C_{r} \to C_{r+1}$, formed by
  pushouts
  \[\cmxymat{
    {\bigoplus_{q \elt \sI}} X_q \ar[r] \ar[d] & C_{r} \ar[d] \\
    {\bigoplus_{q \elt \sI}} Y_q \ar[r] & C_{r + 1} \\
  }\]
  where each $X _q \to Y_q$ is a map in $\sI$.

\item The set $\sI$ is \emph{compact} if, for every map $X \to Y$ in
  $\sI$, the source object, $X$, is small with respect to countable
  colimits.  That is, for every relative $\sI$-cell module $C_{0} \to
  C$ as above, the natural map below is an isomorphism.
  \[
  \colim \Hom_A (X, C_{r}) \fto{\iso} \Hom_A (X, \colim C_{r})
  \]

\item An \emph{$\sI$-cofibration} is a map which satisfies the LLP
  with respect to any map satisfying the RLP with respect to all maps
  in $\sI$.
\end{enumerate}
\end{defn}

\begin{defn}[Cell submodule]
  If $M = \colim M_{r}$ and $L = \colim L_{r}$ are cell
  $A$-modules for which each $L_{r}$ is a submodule of $M_{r}$
  and, for each attaching map $S^q \to L_{r}$, the composite $S^q
  \to L_{r} \subset M_{r}$ is one of the attaching maps for
  $M_{r}$, then $L$ is called a \emph{cell submodule} of $M$.
\end{defn}

\begin{thm}[HELP {\cite[III.2.2]{kriz1995oam}}]\label{HELP}
  Let $L$ be a cell submodule of a cell $A$-module, $M$, and let $e :
  N \to P$ be a quasi-isomorphism of $A$-modules.  Then, given maps
  which make the solid arrow diagram below commute, there are dashed
  lifts which commute with the rest of the diagram.
\[\cmxymat @C=8pt @R=10pt{
  L \ar[rr]^{i_0} \ar[dd] & & L \otimes I \ar[dd]|\hole \ar[dl]_{h} & & L \ar[ll]_{i_1} \ar[dl]_{g} \ar[dd] \\
   & P & & N \ar[ll]_(.42)e & \\
  M \ar[rr]_{i_0} \ar[ur]^{f} & & M \otimes I \ar@{-->}[lu] & & M \ar[ll]^{i_1} \ar@{-->}[lu] \\
}\]
\end{thm}

The following lemma clarifies the relationship between HELP and
quasi-isomorphisms.  It is obvious from \cite[III.2.1]{kriz1995oam},
although they state and prove only one direction.  \autoref{HELP} is
proven by using the relative cell structure $L \to M$ to reduce to
this lemma.

\begin{note}
  For one who compares this lemma with \cite{kriz1995oam}, it may be helpful to
  point out that the grading is cohomological there, so they use $s$
  and $s-1$ where we use $n$ and $n+1$.
\end{note}

\begin{lem}\label{miniHELP}
  For any integer $n$, a map $e:N \to P$ of DG-modules over $A$
  satisfies HELP with respect to the inclusion $S^n \to D^{n+1}$ if
  and only if $e_* : H_*(N) \to H_*(P)$ is a monomorphism in degree
  $n$ and an epimorphism in degree $n+1$.
\end{lem}
\begin{proof}
  Having HELP with respect to $S^n \to D^{n+1}$ means having the
  dotted lifts in any solid-arrow diagram of the type shown below.
  \[
  \cmxymat @C=8pt @R=10pt{
    S^n \ar[rr]^{i_0} \ar[dd] & & S^n \otimes I \ar[dd]|\hole \ar[dl]_-{\tha} & & S^n \ar[ll]_{i_1} \ar[dl]_{z} \ar[dd] \\
    & P & & N \ar[ll]_(.42)e & \\
    D^{n+1} \ar[rr]_{i_0} \ar[ur]^{w'} & & D^{n+1} \otimes I \ar@{-->}[lu]^-{\eta} & & D^{n+1} \ar[ll]^{i_1} \ar@{-->}[lu]^-{w} \\
  }
  \]
  In words (using subscripts to denote degrees of elements, and
  including factors of $(-1)^n$ implicitly where appropriate in our
  correspondence between letters above and letters below), this says
  that given any cycle, $z_n$, in $N$ whose image in $P$ is homologous
  to a boundary, $z_n'$:
  \[
  z_n' = dw_{n+1}' \text{ and } ez_n - z_n' = d\tha_{n+1},
  \]
  then $z_n$ is a boundary in $N$ of some $w_{n+1}$ and, moreover, the
  image of that bounding element in $P$ is homologous to the
  difference between the bounding element for $z_n'$ and the bounding
  element for $ez_n - z_n'$:
  \[
  z_n = dw_{n+1} \text{ and } ew_{n+1} - w_{n+1}' + \tha_{n+1} =
  d\eta_{n+2}.
  \]

  If $e$ has this lifting property, then taking $\tha_{n+1} = 0$ shows
  that $e_*$ is a monomorphism in degree $n$, and taking $z_n,z_n' =
  0$ (so $\tha_{n+1}$ is a cycle in $P$) shows that $e_*$ is an
  epimorphism in degree $n+1$.  For the converse, $e_n$ being a
  monomorphism gives the existence of an element,
  $\wt{w}_{n+1}$ whose boundary is $z_n$, since $ez_n$ is
  homologous to a boundary in $P$.  The existence of $\eta_{n+2}$
  follows because $e_{n+1}$ is an epimorphism and $e\wt{w} -
  w_{n+1}' + \tha_{n+1}$ is a cycle, so there is some cycle
  $\widehat{w}$ in $N$ for which $(e\wt{w} - w_{n+1}' +
  \tha_{n+1}) - e\widehat{w} = d\eta_{n+2}$.  The element $w_{n+1}$ is
  taken to be the difference $\wt{w} - \widehat{w}$.
\end{proof}

As HELP indicates, we use the inclusions $S^q \to D^{q+1}$ and $D^q
\to D^q \otimes I$ to generate the model structure for DG modules over
$A$.  This is formalized by any of several standard results for model
structures, and we quote one such result here.

\begin{thm}[{\cite[4.5.6]{may2006pht}}]
  Let $\sC$ be a bicomplete category with a subcategory of weak
  equivalences (that is, a subcategory containing all isomorphisms in
  $\sC$ and closed under retracts and the two out of three property).
  Let $\sI$ and $\sJ$ be compact sets of maps in $\sC$. If the
  following two conditions hold, then $\sC$ is a compactly generated
  model category, with generating cofibrations $\sI$ and generating
  acyclic fibrations $\sJ$:
\begin{enumerate}
\renewcommand{\theenumi}{\roman{enumi}}
\item (Acyclicity) Every relative $\sJ$-cell complex is a weak
  equivalence.
\item (Compatibility) A map has the RLP with respect to $\sI$ if and
  only if it is a weak equivalence and has the RLP with respect to
  $\sJ$.
\end{enumerate}
\end{thm}

\begin{note}
The term \emph{compactly generated} is a specialization of the notion
of cofibrantly generated for the case that $\sI$ and $\sJ$ are compact
sets of maps (recall \autoref{MSdef}).  It means that the fibrations
are characterized by the RLP with respect to $\sJ$, the acyclic fibrations are
characterized by the RLP with respect to $\sI$, the cofibrations are the
retracts of relative $\sI$-cell complexes, and the acyclic
cofibrations are the retracts of relative $\sJ$-cell complexes.

The main advantages of compact generation over cofibrant generation
are that it does not require one to use the full version of the
small-object argument, but only a small-object argument over countable
colimits, and that it is sufficiently general for many topological and
algebraic applications, including the one which concerns us here.
\end{note}

\sbs{ Application} In our application, $\sC$ will be the category of
(DG) $A$-modules, and the weak equivalences will be the
quasi-isomorphisms.  The set $\sI = \{ S^{n-1} \to D^{n} \mid n \elt
\bZ \}$, and the set $\sJ = \{ D^n \fto{i_0} D^n \otimes I \mid n \elt
\bZ \}$.  By definition, the relative $\sI$-cell modules are the
relative cell $A$-modules, and since $D^n$ and $D^n \otimes I$ are
cell $A$-modules, any relative $\sJ$-cell module is also a cell
$A$-module.  Since $S^{n-1}$ and $D^n$ are finite cell $A$-modules,
the next lemma shows that $\sI$ and $\sJ$ are compact.

\begin{lem}[compactness]\label{cptness}
  If $Z_{0} \to Z_{1} \to Z_{2} \to \cdots \to Z$ is a relative cell
  complex, and if $C$ is a finite cell $A$-module, then the natural
  map below is an isomorphism.
\[
\colim \Hom_A (C, Z_{r}) \fto{\iso} \Hom_A(C, \colim Z_{r})
\]
\end{lem}
\begin{proof}
  Assume first that $C$ is a bounded complex of finitely-generated
  free $A$-modules, with generators $x_1, \ldots, x_n$.  Then an
  A-module map $f: C \to \colim Z_{r}$ is uniquely determined by the
  elements $f(x_i) \elt \colim Z_{r}$.  Since $n$ is finite, there is
  some $s$ such that $f(x_i) \elt im( Z_{s} \to \colim_r Z_{r})$ for
  all $i$.  Hence the lemma holds when $C$ is free; in particular, the
  lemma holds when $C = S^n$ or $C = D^n$.  Now if $M$ is a cell
  complex for which the lemma holds, and $S^n \to M$ is any map of
  $A$-modules, then the pushout of this map along $S^n \to D^{n+1}$ is
  an $A$-module for which the lemma holds.  Since the lemma holds for
  the $A$-module 0, it holds for every finite cell $A$-module.
\end{proof}

\sbs{ Acyclicity}\label{defRet} We note first that the inclusion $i_0
: A \to I$ of free $A$-modules given by $1 \mapsto \lng 0
\rng$ has a deformation retraction, $r$, given by $r \lng I \rng = 0$
and $r(a\lng 0 \rng + b \lng 1 \rng) = a+b$. An explicit homotopy $h:
i_o r \hty \text{id}$ is easily constructed.  For any $A$-module, $M$,
the inclusion $i_0 : M \to M \otimes I$ has a deformation retraction
induced by $r$, and therefore also any map $M \to N$ given by pushout
along $D^n \to D^n \otimes I$ has a deformation retraction.

Since $S^n$ and $S^n \otimes I$ are finite cell $A$-modules, we apply
\autoref{cptness} to see that any relative $\sJ$-cell module is a weak
equivalence.

\sbs{ Compatibility} Let $p : X \to Y$ be a map of $A$-modules.
Assume first that $p$ has the RLP with respect to $\sI$.  Then for
maps
\[
\cmxymat{
  {\bigoplus}S^{q_{i} - 1} \ar[r] \ar[d] & C_{r} \ar[r] \ar[d] & X \ar[d]^-{p} \\
  {\bigoplus}D^{q_i} \ar[r] & C_{r+1} \ar[r] & Y }
\] 
where $C_{r+1}$ is a pushout, we have a lift $\oplus D^{q_i} \to X$
and hence a lift $C_{r+1} \to X$.  Therefore $p$ lifts with respect to
any relative ($\sI$-)cell module, and in particular has the RLP with
respect to $\sJ$.  Moreover, $p$ has the RLP with respect to all maps
$0 \to S^n$ and $S^n \otimes I \to S^n$, and hence $p$ is a weak
equivalence.

Now suppose only that $p$ is a weak equivalence and has the RLP with
respect to $\sJ$.  Being a weak equivalence, $p$ satisfies the
homotopy extension and lifting property (\autoref{HELP}).  To see that
this implies the result, note first that there is an isomorphism of
free $A$-modules ${I} \iso {D^1} \oplus {S^0}$ given by changing
basis in degree 0, and hence a projection ${I} \to {D^1}$ which
equalizes the two inclusions $i_0$ and $i_1 : {S^0} \to {I}$.
Moreover, the composite of either inclusion with this projection is
the standard inclusion ${S^0} \to {D^1}$.  We tensor with ${S^n}$ and
use the isomorphism ${D^{n+1}} \iso {S^n} \otimes {D^1}$ to define a
map ${S^n} \otimes {I} \to {S^n} \otimes {D^1} \iso {D^{n+1}}$
equalizing the inclusions $i_0$ and $i_1: {S^n} \to {S^n} \otimes {I}$
and such that either composite ${S^n} \to {S^n} \otimes {I} \to
{D^{n+1}}$ is the standard inclusion.  In other words, the diagram
below commutes.
\[
\cmxymat @C=8pt @R=10pt{
  S^n \ar[rr]^{i_0} \ar[dd] & & S^n \otimes I \ar[dl] & & S^n \ar[ll]_{i_1} \ar@{=}[dl] \\
  & D^{n+1} & & S^n \ar[ll] & \\
  D^{n+1} \ar@{=}[ur] \\
}
\]

Now given the following commuting square of DG modules over $A$,
\[\cmxymat{
  S^n \ar[r] \ar[d] & X \ar[d]^-{p} \\
  D^{n+1} \ar[r] & Y
}\]
we produce a lift via the commuting diagram below, where the map $S^n
\otimes I \to Y$ is the composite $S^n \otimes I \to D^{n+1} \to Y$.
\[
\cmxymat @C=8pt @R=10pt{
  S^n \ar[rr]^{i_0} \ar[dd] & & S^n \otimes I \ar[dd]|\hole \ar[dl] & & S^n \ar[ll]_{i_1} \ar[dl] \ar[dd] \\
  & Y & & X \ar[ll]_(.42){p} & \\
  D^{n+1} \ar[rr]_{i_0} \ar[ur] & & D^{n+1} \otimes I \ar@{-->}[lu] \ar@{.>}[ru]^{\ell} & & D^{n+1} \ar[ll]^{i_1} \ar@{-->}[lu] \\
}
\] 
The dashed lifts follow from HELP (\autoref{HELP}), and the dotted
lift of these, $\ell$, exists because $p$ has the RLP with respect to
$\sJ$.  The composite $\ell i_0$ is the desired lift for the square
above.

\sbs{ Further Structure}
Here we document some facts about the model structure described above.

\begin{rmk}\label{nothingnew}
  This model structure is the same as the standard model structure for
  chain complexes over a ring.  Hovey's description
  \cite[2.3.3]{hovey1999mc} of the standard model structure for
  DG-modules over $A$ when $A$ is a ring (i.e. chain complexes over
  the ring $A$) has the same weak equivalences and generating
  cofibrations, $\sI$, as above; the generating acyclic cofibrations
  are $\sJ' = \{ 0 \to D^n \}$.  It is clear that the cofibrations of
  these two model structures are the same, and the squares below show
  that the fibrations of these two structures are also the same: RLP
  with respect to $\sJ$ is equivalent to RLP with respect to
  $\sJ'$.\qedhere
  \[
  \cmxymat{
    0 \ar[r] \ar[d] & D^n \ar[r] \ar[d] & 0 \ar[d]\\
    D^{n+1} \ar[r] & D^n \otimes I \ar[r] & D^{n+1} }
  \]
\end{rmk}

\begin{prop}
  All $A$-modules are fibrant.
\end{prop}
\begin{proof}
  The inclusion $D^n \to D^n \otimes I$ has a section, hence the map
  $M \to 0$ has the RLP with respect to all generating acyclic
  cofibrations for any $M$.
\end{proof}

\begin{prop}[{\cite[III.4.1]{kriz1995oam}}]\label{cofib-monAx}
  If $X \to Y$ is a weak equivalence of \emph{left} $A$-modules and
  $M$ is a cofibrant (right) $A$-module, then
  \[
  M \otimes_A X \to M \otimes_A Y
  \]
  is a weak equivalence.
\end{prop}

\section{Base Change for DG Algebras}\label{base-change} 
In this section, we describe general results regarding change of base
DG $k$-algebra.  Suppose that $A$ and $B$ are DG $k$-algebras for a
commutative ring, $k$, and suppose $f : A \to B$ is a map of DG
$k$-algebras.  There are two natural pull-backs of $B$ to the category
of DG $A$-modules: let $_A B_B \elt DG_k(B,A)$ denote $B$ with the
action of $A$ on the left via $f$, and let $_B B_A \elt DG_k(A,B)$
denote $B$ with the action of $A$ on the right via $f$.  Then the
$A$-$A$ bimodule obtained from $B$ with $A$ acting on both sides by
$f$ is given as $({_A B_B}) {\otimes_B} ({_B B_A}) = ({_A B_B})
{\odot} ({_B B_A}) \elt DG_k(A,A)$.  The map $f$ can be regarded as a
2-cell
\[
A \fto{f} {_A B_A} = ({_A B_B}) {\odot} ({_B B_A}).
\]
The multiplication for $B$ gives a 2-cell in $DG_k(B,B)$
\[
({_B B_A}) \odot ({_A B_B}) = ({_B B_A}) \otimes_A ({_A B_B}) \to B
\]
and the duality relations hold, making $({_A B_B} , {_B B_A})$ a dual
pair.  Hence we have an adjoint pair of strong transformations
\[
\text{extension of scalars: }
f_! = - \odot {_A B_B} : DG_k(A,-) \to DG_k(B,-) 
\]
and
\[
\text{restriction of scalars: }
f^* = - \odot {_B B_A} \iso {_A B_B} \rt - : DG_k(B,-) \to DG_k(A,-)
\]
The transformation $f^*$ is right adjoint to $f_!$, but since $f^*$ is
itself a strong transformation, it also has its own right adjoint,
\[
f_* = {_B B_A} \rt - : DG_k(A,-) \to DG_k(B,-).
\]

\sbs{ Local model structure} Each 1-cell category $DG_k(A,B)$ has a
model structure, described by applying the theory of \autoref{Model}
to the DG $k$-algebra $A \otimes_k B^{op}$. We refer to this as a
local model structure for the bicategory $DG_k$, meaning simply a
model structre on each 1-cell category.  The generating cofibrations
and acyclic cofibrations of $DG_k(A,B)$ are denoted by $\sI(A,B)$ and
$\sJ(A,B)$, respectively, and the results below describe the behavior
of the base-change transformations above with respect to this local
model structre.

\begin{notation}
  In contrast with \autoref{Model}, here we let $S^n$, $D^m$, and $I$
  denote the corresponding chain complexes over $k$, and we let
  $\otimes$ denote $\otimes_k$.  For any chain complex, $M$, over $k$,
  we let $_B M_A$ denote the DG $(B,A)$-bimodule $B \otimes_k M
  \otimes_k A$.  So $_B M_A \elt DG_k(A,B)$.
\end{notation}

\begin{prop}[Push-out Products]\label{pop}
  The local model structure on each 1-cell category $DG_k(A,B)$ is
  compatible with $\odot$-composition of 1-cells in the following
  sense: If $i$ and $j$ are generating cofibrations, then their
  pushout-product is a cofibration, and if one of $i$ or $j$ is a
  generating acyclic cofibration and the other is a generating
  cofibration, then their pushout-product is an acyclic cofibration.
\end{prop}
\begin{proof}
  Suppose first that $i: S^q \to D^{q+1}$ and $j: S^r \to D^{r+1}$ are
  generating cofibrations in $Ch(k)$.  If we denote by $\bP$ the
  pushout below, then the pushout product of $i$ and $j$ is the
  induced map $\bP \to D^{q+1} \otimes D^{r+1}$.
  \[
  \cmxymat @C=6pt @R=12pt{
    S^q \otimes S^r \ar[r] \ar[d] & S^q \otimes D^{r+1} \ar[d] & \\
    D^{q+1} \otimes S^r \ar[r] & {\bP} \ar@{-->}[dr] & \\
    & & D^{q+1} \otimes D^{r+1} }
  \]
  This map can be obtained explicitly by attaching a cell of dimension
  $q + r + 2$ to $\bP$, and hence is a cofibration.  Likewise, if $j$
  is taken to be a generating acyclic cofibration $D^{r+1} \to D^{r+1}
  \otimes I$, the pushout product can be seen explicitly to be a
  cofibration.  Since extension of scalars to an arbitrary DG
  $k$-algebra preserves cofibrations, the pushout products over
  $\odot$ are still cofibrations.  When $i$ or $j$ is taken to be a
  generating acyclic cofibration, we see that the pushout product is
  still acyclic by recalling the deformation retraction $D^{r+1}
  \fto{i_0} D^{r+1} \otimes I \fto{\hty} D^{r+1}$ of \autoref{defRet}.
  Extending to another DG $k$-algebra, we still have these deformation
  retractions, so the pushout products over $\odot$ remain weak
  equivalences.
\end{proof}

\begin{prop}\label{Bchange}
  If $f : A \to B$ is a map of DG $k$-algebras, then the adjoint
  pair $(f_!,f^*)$ is a Quillen adjoint pair for each $C$.
  \[\cmxymat{
    DG_k(A,C) \ar@<.4ex>^-{f_!}[r] & DG_k(B,C) \ar@<.4ex>^-{f^*}[l]\\
  }\]
\end{prop}
\begin{proof}
  This follows from the observation that, for a chain complex $M \elt
  Ch(k)$, 
  \[
  f_!(_C M_A) = {_C M_A} \odot {_A B_B} \iso {_C M_B}.
  \]
  Hence $f_!$ induces an isomorphism of sets $\sI(A,C) \iso \sI(B,C)$
  and $\sJ(A,C) \iso \sJ(B,C)$, where $\sI$ and $\sJ$ are the
  generating cofibrations and acyclic cofibrations for $DG_k(A,C)$ and
  $DG_k(B,C)$.
\end{proof}

\begin{rmk}
  A similar statement for $(f^*,f_*)$ is not true unless $_A B_A$ is
  cofibrant as an $A$-module, since otherwise $f^*$ does not preserve
  cofibrations in general.
\end{rmk}

\begin{lem}[$f^*$ creates weak equivalences]\label{createWE}
  If $e: X \to Y$ is a map of 1-cells in $DG_k(B,C)$ for which $f^* e: 
  f^*X \to f^*Y$ is a weak equivalence, then the original map $e: X
  \to Y$ is a weak equivalence.
\end{lem}
\begin{proof}
  If $L \to M$ is a cofibration in $Ch(k)$, we have noted already that
  the induced map ${_C L_B} \to {_C M_B}$ is isomorphic to $f_!  \big(
  {_C L_A} \to {_C M_A} \big)$.  To show that $e: X \to Y$ is a weak
  equivalence, it suffices to show that $e$ has HELP with respect to
  maps of this form (\autoref{miniHELP}).  Consider the adjoint
  lifting diagrams below.
  \[
  \cmxymat @C=8pt @R=16pt{
    f_!{_C L_A} \ar[rr] \ar[dd] & & f_!{_C (L \otimes I)_A} \ar[dd]|\hole \ar[dl] & & f_!{_C L_A} \ar[ll] \ar[dl] \ar[dd] \\
    & Y & & X \ar[ll]_(.42)e & \\
    f_!{_C M_A} \ar[rr] \ar[ur] & & f_!{_C (M \otimes I)_A} \ar@{-->}[lu] & & f_!{_C M_A} \ar[ll] \ar@{-->}[lu] \\
  }\hspace{2pc} \cmxymat @C=8pt @R=15pt{
    {_C L_A} \ar[rr] \ar[dd] & & {_C (L \otimes I)_A} \ar[dd]|\hole \ar[dl] & & {_C L_A} \ar[ll] \ar[dl] \ar[dd] \\
    & f^*Y & & f^*X \ar[ll]_(.42){f^*e} & \\
    {_C M_A} \ar[rr] \ar[ur] & & {_C (M \otimes I)_A} \ar@{-->}[lu] & & {_C M_A} \ar[ll] \ar@{-->}[lu] \\
  }
  \]
  Lifts exist on the right by hypothesis, and hence on the left by
  adjunction.
\end{proof}

\begin{prop}\label{Bchange-we}
  If $f$ above is a weak equivalence, then the Quillen pair $(f_!,f^*)$
  is a Quillen equivalence for all $C$.
\end{prop}
\begin{proof}
  By \cite[Prop. 1.3.13]{hovey1999mc} it suffices to prove that, for
  cofibrant 1-cells $M \elt DG_k(A,C)$ and $N \elt DG_k(B,C)$, the
  composites
  \[
  M \to f^*f_!M = M \odot {_A B_B} \odot {_B B_A} \iso M \odot {_A
    B_A}
  \]
  and
  \[
  f_!Q(f^*N) = Q(N \odot {_B B_A}) \odot {_A B_B} \to N \odot {_B B_A}
  \odot {_A B_B} \to N
  \]
  are weak equivalences.  The functor $Q$ is cofibrant replacement; no
  fibrant replacement is necessary since every object is fibrant.

  That the top composite is a weak equivalence follows because $M$ is
  a cofibrant $(C,A)$-bimodule and $A \to {_A B_A}$ is a weak
  equivalence of $(A,A)$-bimodules (\autoref{cofib-monAx}).  To see
  that the bottom composite is a weak equivalence, we consider the
  composite
  \[
  Q(N \odot {_B B_A}) \to Q(N \odot {_B B_A}) \odot {_A B_B} \odot {_B
    B_A} \to N \odot {_B B_A}
  \]
  where the first map is the unit of the adjunction, and the second is
  $f^*$ applied to the composite above.  This total composite is the
  cofibrant replacement map $Q(N \odot {_B B_A}) \to N \odot {_B
    B_A}$, and hence is a weak equivalence by definition.  The first
  map in this composite is a weak equivalence because $Q(N \odot {_B
    B_A})$ is cofibrant (as above), and hence by the two-out-of-three
  property, $f^*\big( Q(N \odot {_B B_A}) \odot {_A B_B} \to N \big)$
  is a weak equivalence.  By \autoref{createWE} then, the map $Q(N
  \odot {_B B_A}) \odot {_A B_B} \to N$ is a weak equivalence.
\end{proof}

\begin{cor}
  For $f : A \fto{\hty} B$ as above, the dual pair $({_A B_B} , {_B B_A})$ is
  invertible when considered as a pair of 1-cells in the derived
  categories $\sD_k(B,A)$ and $\sD_k(A,B)$, respectively.
\end{cor}

\sbs{ Duality in $DG_k$ and $\sD_k$}
\begin{lem}\label{dlz-cofib}
  If $M$ is right-dualizable in $DG_k(A,B)$, then $M$ is a retract of a
  finite free (right-)DG-module over $A \otimes_k B^{op}$.
\end{lem}
\begin{proof}
  Since $M$ is right-dualizable, the coevaluation map $\nu : M \odot
  (M \rt A) \to M \rt M$ is an isomorphism, and hence there is a map
  $B \to M \odot (M \rt A)$ lifting the unit $B \to M \rt M$.  We let
  $\SI_i (m_i \otimes \phz_i)$ denote the image of the unit, $1_B$,
  under this map, where $m_i \elt M$, $\phz_i \elt M \rt A =
  \Hom_A(M,A)$, and the sum has only finitely many terms.

  We now show that there is an $A \otimes_k B^{op}$-module map $p$,
  with a section $\tilde\phz$, making $M$ a retract of a
  finitely-generated free DG-module, where each $e_i$ is a generator
  of degree $|e_i| = |m_i|$:

  \[\cmxymat@C=1pc{
    **[l]{\bigoplus_i} (A \otimes_k B^{op} \cdot e_i) \ar[r]_-p & M. \ar@/_1pc/_-{\tilde\phz}[l]
  }\]

  The map $p$ is defined by $p(a \otimes b \cdot e_i) = b \cdot m_i
  \cdot a$, and the section $\tilde\phz$ is defined by $\tilde\phz(m)
  = \SI_i \phz_i(m) \otimes 1_B$.  It is easy to see that $\tilde\phz$
  is a section for $p$ by making use of the fact that $id_M = \nu(
  \SI_i m_i \otimes \phz_i) = \SI_i m_i \cdot \phz_i$,
\end{proof}

\begin{lem}\label{retDlz}
  Let $M \elt \sD_k(A,B)$ and suppose $M$ is a retract of a finite
  cell $(B,A)$-bimodule.  Then $M : A \to B$ is (right-)dualizable in
  $\sD_k$ and therefore the coevaluation $M \odot (M \rt A) \to M \rt
  M$ is an isomorphism in $\sD_k$.
\end{lem}
\begin{note}
  Since we are working in the derived bicategory, $\sD_k$, the
  source-homs, $\rt$, above are understood to be the derived homs.
  Since $M$ is cofibrant, the derived and underived homs are equal on
  $M$.
\end{note}
\begin{proof}
  One characterization of duality is that the map induced by
  evaluation $\sD_k[W,Z \odot (M \rt A)] \to \sD_k[W \odot M, Z]$ be
  an isomorphism for all 1-cells $W$, $Z$ with appropriate source and
  target.  From this point of view, the five lemma shows that the full
  subcategory of dualizable objects in $\sD_k(A,B)$ is a thick
  subcategory (see, for example, \cite[16.8.1]{may2006pht}).  Since
  the pushouts which build cell modules are examples of exact
  triangles in $\sD_K$, the result follows by noting that the spheres
  and disks, $S^q$ and $D^{q}$, are dualizable.

  The coevaluation map, $\nu$, is defined as the adjoint to $M \odot
  (M \rt A) \odot M \to M \odot A \iso M$, induced by the evaluation
  map, so if $M$ is dualizable in the sense described above, then
  taking $W = M \rt M$ and $Z = M$ produces the inverse to
  coevaluation.
\end{proof}

\begin{lem}\label{retDlz-converse}
  Let $M: A \to B$ be a 1-cell in $\sD_k$, and suppose the
  coevaluation $M \odot (M \rt A) \to M \rt M$ is an isomorphism.
  Then $M$ is (quasi-)isomorphic to a retract of a finite cell
  $(B,A)$-module.
\end{lem}
\begin{proof}
  Following the usual argument, we implicitly take a cofibrant
  replacement for $M$ as a $(B,A)$-bimodule, $\colim M_{r} \fto{\hty}
  M$.  The inverse to coevaluation, composed with the monoid map $B
  \to M \rt M$ gives $\eta: B \to M \odot (M \rt A)$.  Since $- \odot
  (M \rt A)$ preserves colimits, and since $B$ is compact in
  $\sD_k(B,B)$, this map factors through some finite stage of $\colim
  (M_{r} \odot (M \rt A))$, and we have a lift in the diagram below
  for some $r$.

  \[
  \cmxymat{
    & & M_{r} \odot (M \rt A) \odot M \ar[rr]\ar[d] & & M_{r} \ar[d]\\
    M \iso B \odot M \ar[rr]^-{\eta \odot \id} \ar@{-->}[rru] & & M
    \odot (M \rt A) \odot M \ar[rr]^-{\id \odot \text{eval}} & & M }
  \]
  The bottom composite is the identity, and we see that $M$ is a
  retract of $M_r$.
\end{proof}

\bibliographystyle{amsalpha}
%\nocite{*}
%\bibliography{../NilesRefs}
\bibliography{Morita-derived.bbl}

\providecommand{\bysame}{\leavevmode\hbox to3em{\hrulefill}\thinspace}
\providecommand{\MR}{\relax\ifhmode\unskip\space\fi MR }
\providecommand{\MRhref}[2]{%
  \href{http://www.ams.org/mathscinet-getitem?mr=#1}{#2}
}
\providecommand{\href}[2]{#2}
\begin{thebibliography}{LMS86}

\bibitem[Bro03]{brouwer2003bam}
R.M. Brouwer, \emph{{A bicategorical approach to Morita equivalence for rings
  and von Neumann algebras}}, arXiv:math/0301353v1 [math.OA] (2003).

\bibitem[BV02]{borceux2002ac}
F.~Borceux and E.~Vitale, \emph{{Azumaya Categories}}, Applied Categorical
  Structures \textbf{10} (2002), no.~5, 449--467.

\bibitem[DS07]{dugger2007ted}
D.~Dugger and B.~Shipley, \emph{{Topological equivalences for differential
  graded algebras}}, Advances in Mathematics \textbf{212} (2007), 37--61.

\bibitem[FPP75]{fisherpalmquist1975mce}
J.~Fisher-Palmquist and P.H. Palmquist, \emph{{Morita Contexts of Enriched
  Categories}}, Proceedings of the American Mathematical Society \textbf{50}
  (1975), no.~1, 55--60.

\bibitem[Hov99]{hovey1999mc}
M.~Hovey, \emph{{Model Categories}}, American Mathematical Society, 1999.

\bibitem[Kel94]{keller1994ddc}
B.~Keller, \emph{{Deriving DG categories}}, Annales Scientifiques de
  l'{\'E}cole Normale Sup{\'e}rieure S{\'e}r. 4 \textbf{27} (1994), no.~1,
  63--102.

\bibitem[KM95]{kriz1995oam}
I.~Kriz and J.P. May, \emph{{Operads, algebras, modules and motives}},
  Asterisque, vol. 233, Soci{\'e}t{\'e} Math{\'e}matique de France, 1995.

\bibitem[KZ98]{konig1998deg}
S.~K{\"o}nig and A.~Zimmermann, \emph{{Derived equivalences for group rings.
  With contributions by B. Keller, M. Linckelmann, J. Rickard and R.
  Rouquier}}, Lecture Notes in Mathematics, vol. 1685, Springer-Verlag, Berlin,
  1998.

\bibitem[Lac07]{lack20072cc}
S.~Lack, \emph{{A 2-Categories Companion}}, arXiv:math/0702535v1 [math.CT]
  (2007).

\bibitem[Lam99]{lam1999lmr}
T.Y. Lam, \emph{{Lectures on Modules and Rings}}, Graduate Texts in
  Mathematics, vol. 189, Springer, 1999.

\bibitem[Lei98]{leinster1998bb}
T.~Leinster, \emph{{Basic Bicategories}}, arXiv:math/9810017v1 [math.CT]
  (1998).

\bibitem[LMS86]{lewis1986esh}
L.G. Lewis, J.P. May, and M.~Steinberger, \emph{{Equivariant stable homotopy
  theory}}, Lecture Notes in Mathematics, vol. 1213, Springer, 1986.

\bibitem[MS06]{may2006pht}
J.P. May and J.~Sigurdsson, \emph{{Parametrized Homotopy Theory}}, Mathematical
  Surveys and Monographs, vol. 132, American Mathematical Society, 2006.

\bibitem[M{\"u}g03]{muger2003sca}
M.~M{\"u}ger, \emph{{From subfactors to categories and topology I: Frobenius
  algebras in and Morita equivalence of tensor categories}}, Journal of Pure
  and Applied Algebra \textbf{180} (2003), 81--157.

\bibitem[Nee01]{neeman2001tc}
A.~Neeman, \emph{{Triangulated Categories}}, Princeton University Press, 2001.

\bibitem[Ric89]{rickard1989mtd}
J.~Rickard, \emph{{Morita Theory for Derived Categories}}, Journal of the
  London Mathematical Society \textbf{2} (1989), no.~3, 436.

\bibitem[Ric91]{rickard1991ded}
\bysame, \emph{{Derived Equivalences As Derived Functors}}, Journal of the
  London Mathematical Society \textbf{2} (1991), no.~1, 37.

\bibitem[Sch04]{schwede2004mta}
S.~Schwede, \emph{{Morita theory in abelian, derived and stable model
  categories}}, Structured ring spectra, London Math. Soc. Lecture Note Ser,
  vol. 315, Cambridge University Press, 2004, pp.~33--86.

\bibitem[Shi06]{shipley2006mts}
B.~Shipley, \emph{{Morita Theory In Stable Homotopy Theory}}, Handbook on
  Tilting Theory, London Math. Soc. Lecture Note Ser., vol. 332, Cambridge
  University Press, 2006, pp.~393--409.

\bibitem[SS03]{schwede2003smc}
S.~Schwede and B.~Shipley, \emph{{Stable model categories are categories of
  modules}}, Topology \textbf{42} (2003), no.~1, 103--153.

\bibitem[Str80]{street1980fb}
R.~Street, \emph{{Fibrations in bicategories}}, Cahiers de Topologie et
  Geometrie Differentielle \textbf{21} (1980), no.~2, 111--160.

\bibitem[Szl04]{szlachanyi2004mme}
K.~Szlach{\'a}nyi, \emph{{Monoidal Morita equivalence}}, arXiv:math/0410407v1
  [math.QA] (2004).

\bibitem[To{\"e}07]{toen2007htd}
B.~To{\"e}n, \emph{{The homotopy theory of dg-categories and derived Morita
  theory}}, Inventiones Mathematicae \textbf{167} (2007), no.~3, 615--667.

\bibitem[Wat60]{watts1960ics}
C.E. Watts, \emph{{Intrinsic Characterizations of Some Additive Functors}},
  Proceedings of the American Mathematical Society \textbf{11} (1960), no.~1,
  5--8.

\end{thebibliography}

\end{document}